\documentclass[11pt]{amsart}

\usepackage{epigamath}

\usepackage[english]{babel}

\numberwithin{equation}{section}

\usepackage{enumitem}

\newtheorem{theorem}{Theorem}[section]
\newtheorem{lemma}[theorem]{Lemma}
\newtheorem{lemma-def}[theorem]{Lemma-Definition}
\newtheorem{prop}[theorem]{Proposition}
\newtheorem{cor}[theorem]{Corollary}
\newtheorem{conj}[theorem]{Conjecture}

\theoremstyle{definition}
\newtheorem{convention}[theorem]{Convention}
\newtheorem{hypothesis}[theorem]{Hypothesis}
\newtheorem{notation}[theorem]{Notation}
\newtheorem{defn}[theorem]{Definition}

\theoremstyle{remark}
\newtheorem{remark}[theorem]{Remark}

\newcommand{\AAA}{\mathbb{A}}
\newcommand{\CC}{\mathbb{C}}

\newcommand{\PP}{\mathbb{P}}
\newcommand{\QQ}{\mathbb{Q}}
\newcommand{\RR}{\mathbb{R}}
\newcommand{\ZZ}{\mathbb{Z}}
\newcommand{\calC}{\mathcal{C}}
\newcommand{\calE}{\mathcal{E}}
\newcommand{\calF}{\mathcal{F}}
\newcommand{\calG}{\mathcal{G}}
\newcommand{\calL}{\mathcal{L}}
\newcommand{\calO}{\mathcal{O}}

\newcommand{\dual}{\vee}

\DeclareMathOperator{\Ext}{Ext}

\DeclareMathOperator{\FIsoc}{\mathbf{F-Isoc}}
\DeclareMathOperator{\Gal}{Gal}
\DeclareMathOperator{\GL}{GL}

\DeclareMathOperator{\rank}{rank}

\DeclareMathOperator{\Spec}{Spec}

\DeclareMathOperator{\Swan}{Swan}

\DeclareMathOperator{\Weil}{\mathbf{Weil}}

\EpigaVolumeYear{6}{2022} \EpigaArticleNr{20} \ReceivedOn{October 4,
2020}
\InFinalFormOn{June 21, 2022}
\AcceptedOn{July 28, 2022}

\title{\'Etale and crystalline companions, I}
\titlemark{\'Etale and crystalline companions, I}

\author{Kiran S. Kedlaya}
\address{Department of Mathematics,
University of California San Diego,
La Jolla, CA 92093-0112}
\email{kedlaya@ucsd.edu}

\authormark{K.~S.~Kedlaya}

\AbstractInEnglish{Let $X$ be a smooth scheme over a finite field of characteristic $p$.
Consider the coefficient objects of locally constant rank on $X$ in $\ell$-adic Weil cohomology: these are lisse Weil sheaves in \'etale cohomology when $\ell \neq p$, and overconvergent $F$-isocrystals in rigid cohomology when $\ell=p$. Using the Langlands correspondence for global function fields in both the \'etale and crystalline settings
(work of Lafforgue and Abe, respectively), one sees that on a curve, any coefficient object in one category has ``companions'' in the other categories with matching characteristic polynomials of Frobenius at closed points.
A similar statement is expected for general $X$; building on work of Deligne, Drinfeld showed that any \'etale coefficient object has \'etale companions. We adapt Drinfeld's method to show that any crystalline coefficient object has \'etale companions; this has been shown independently by Abe--Esnault.
We also prove some auxiliary results relevant for the construction of crystalline companions of \'etale coefficient objects; this subject will be pursued in a subsequent paper.}



\acknowledgement{The author was supported by NSF (grants DMS-1501214, DMS-1802161, DMS-2053473), UCSD (Warschawski Professorship), and IAS (Visiting Professorship 2018--2019); thanks also to BIRS for hosting the workshop ``$p$-adic Cohomology and Arithmetic Applications'' in October 2017 and to IM PAN (Simons Foundation grant 346300, Polish MNiSW fund 2015--2019) for its hospitality during September 2018.}

\begin{document}



\maketitle

\begin{prelims}

\DisplayAbstractInEnglish

\bigskip

\DisplayKeyWords

\medskip

\DisplayMSCclass







\end{prelims}


\newpage

\setcounter{tocdepth}{1}

\tableofcontents


\section{Introduction}

\subsection{Coefficients, companions, and conjectures}

Throughout this introduction (but not beyond; see \S\ref{sec:companions}), let $k$ be a finite field of characteristic $p$
and let $X$ be a smooth scheme over $k$. When studying the cohomology of motives over $X$, one typically fixes a prime $\ell \neq p$
and considers \'etale cohomology with $\ell$-adic coefficients; in this setting,
the natural coefficient objects of locally constant rank are the \emph{lisse Weil $\overline{\QQ}_\ell$-sheaves}. However, \'etale cohomology with $p$-adic coefficients behaves poorly in characteristic $p$; for $\ell=p$, the correct choice for a Weil cohomology with $\ell$-adic coefficients is Berthelot's \emph{rigid cohomology}, wherein
the natural coefficient objects of locally constant rank are the \emph{overconvergent $F$-isocrystals}. 
For the purposes of this introduction, the exact nature of these objects is not material; we suggest
our recent survey \cite{kedlaya-isocrystals} as a starting point for this topic.

For $\calE$ a coefficient object on $X$ and $x$ a closed point of $X$, let $\calE_x$ denote the fiber of $\calE$ at $x$. This is a finite-dimensional $\overline{\QQ}_\ell$-vector space carrying an invertible linear action by (geometric) Frobenius, whose reverse characteristic polynomial we denote by $P(\calE_x, T)$. 

For the remainder of this introduction, let $\ell, \ell'$ be two primes (which may or may not equal each other and/or $p$), and fix an isomorphism of the algebraic closures of $\QQ$ within $\overline{\QQ}_\ell$ and $\overline{\QQ}_{\ell'}$ (which may be nontrivial in case $\ell = \ell'$).
We say that an $\ell$-adic coefficient object $\calE$ and an $\ell'$-adic coefficient object $\calE'$ are \emph{companions} if  for all $x$,
the characteristic polynomials of Frobenius on $\calE_x$ and $\calE'_x$ coincide;
given this relation, each object determines the other up to semisimplification (see Theorem~\ref{T:chebotarev}).

In his work on the Weil conjectures, Deligne introduced a far-reaching set of conjectures
\cite[Conjecture~1.2.10]{deligne-weil2} motivated by the idea that coefficient objects should only exist for ``geometric reasons'' (more on which below); we state here a lightly modified version of this conjecture. 
(Compare \cite[Conjecture~1.1]{cadoret}.)

\begin{conj}[after Deligne] \label{conj:deligne}
Let $\calE$ be an $\ell$-adic coefficient object which is irreducible with determinant of finite order.
(Remember that $\ell$ may or may not equal $p$.)
\begin{enumerate}
\item[(i)]
$\calE$ is \emph{pure of weight $0$}: for every algebraic embedding of $\overline{\QQ}_\ell$ into $\CC$,
for all $x$, the roots of $P(\calE_x, T)$ in $\CC$ all have complex absolute value $1$.
$($One can avoid having to embed $\overline{\QQ}_\ell$ into $\CC$ by reading (ii) first.$)$
\item[(ii)]
For some number field $E$, 
$\calE$ is \emph{$E$-algebraic}: for all $x$, we have $P(\calE_x, T) \in E[T]$.
\item[(iii)]
$\calE$ is \emph{$p$-plain}: for all $x$, the roots of $P(\calE_x, T)$ have trivial $\lambda$-adic valuation at all finite places of $\lambda$ of $E$ not lying above $p$.
$($The terminology here follows D'Addezio \cite[Definition~2.0.2]{daddezio}; the terminology used by Chin \cite{chin2} is \emph{plain of characteristic $p$}.$)$

\item[(iv)]
For every place $\lambda$ of $E$ above $p$, for all $x$, the roots of $P(\calE_x, T)$ have $\lambda$-adic valuation at most
$\frac{1}{2} \rank(\calE)$ times the valuation of $\#\kappa(x)$ $($the order of the residue field$\,)$.
\item[(v)]
For every prime $\ell' \neq p$, there exists an $\ell'$-adic coefficient object $\calE'$ which is irreducible with determinant of finite order and is a companion of $\calE$.
\item[(vi)]
As in (v), but with $\ell' = p$. $($We separate this case to maintain parallelism with \cite[Conjecture~1.2.10]{deligne-weil2}.$)$
\end{enumerate}
\end{conj}

By way of comparison with \cite[Conjecture~1.2.10]{deligne-weil2}, we restrict to $X$ smooth (not just normal) but do not require $X$ to be connected; we allow $\ell=p$ in the initial data (which is relevant for our main result); and 
in (vi) we incorporate Crew's proposal to interpret Deligne's undefined phrase ``petits camarades cristallins'' to mean ``overconvergent $F$-isocrystals'' (a concept which did not exist at the time of \cite{deligne-weil2};
see \cite[Conjecture~4.13]{crew-mono} and \cite[Conjecture~D]{abe-crelle}).

The purpose of this paper is to consolidate and extend our knowledge towards this conjecture; this includes extensive work by other authors (attributed in this introduction) as well as original results as noted. 
To summarize, when $\dim(X) = 1$ everything is known; when $\dim(X) > 1$ everything except part (vi) is known;
and part (vi) will be addressed in a subsequent paper \cite{kedlaya-companions2}.

\subsection{Curves and the Langlands correspondence}

A key step in the proof of Conjecture~\ref{conj:deligne}
is the construction of the $\ell$-adic Langlands correspondence for the group $\GL_n$ over a global function field of positive characteristic. This was already anticipated by Deligne in \cite{deligne-weil2}; at the time, the Langlands correspondence was available for $n=2$, $\ell \neq p$ by the work of Drinfeld \cite{drinfeld}.
Deligne observed that Drinfeld's construction not only established the requisite bijection between $\ell$-adic coefficient objects on a curve and automorphic representations on the associated function field,
but also realized the latter as the realizations of certain motives appearing in moduli spaces of shtukas; this then provides a geometric origin for these coefficient objects.

Following Drinfeld's approach,
L. Lafforgue \cite{lafforgue} established the Langlands correspondence for $\GL_n$ 
for arbitrary $n$ for $\ell \neq p$. The case $\ell=p$ had to be excluded at the time due to the limited foundational development of $p$-adic cohomology at the time; this issue has subsequently been remedied (see \cite{kedlaya-isocrystals} for a summary), and this enabled Abe \cite{abe-companion} to extend Lafforgue's work to the case $\ell=p$. 
This immediately implies all of Conjecture~\ref{conj:deligne} in the case where $\dim(X) = 1$
\cite[Th\`eor\'eme~VII.6]{lafforgue}, \cite[\S 4.4]{abe-companion}.

To sum up, we have in hand the following statement, which we will use as a black box in what follows;
see \S\ref{subsec:one-dim} for the proof.

\begin{theorem} \label{T:one-dim case}
Conjecture~\ref{conj:deligne} holds in all cases where $\dim(X) = 1$.
\end{theorem}

\subsection{Avoiding residual representations}

In order to make further progress without losing the case $\ell = p$, one must replace several key arguments in \'etale cohomology that make direct use of the representation-theoretic interpretation of \'etale coefficient objects.
These cannot be used for $\ell = p$ because only a rather small subcategory of coefficient objects can be described in terms of representations of the \'etale fundamental group (namely the \emph{unit-root} objects; see \S\ref{sec:Newton poly}). Such arguments can sometimes be circumvented using cohomological considerations, as in the following key examples.

\begin{enumerate}[label=$\bullet$]
\item
For $\ell \neq p$, the uniqueness of $\ell$-adic companions up to semisimplification is usually obtained by applying
the Chebotar\"ev density theorem to mod-$\ell^n$ representations
(as in the proof of \cite[Proposition~VI.11]{lafforgue}).
For $\ell = p$, we use instead an argument of Tsuzuki based on Deligne's theory of weights (see Theorem~\ref{T:chebotarev}). This result is already used crucially
in the work of Abe \cite[Proposition~A.4.1]{abe-companion},
and is akin to the argument of Deligne described below to control coefficient fields
(see especially \cite[Proposition~2.5]{deligne-finite}).
\item
The extension of Conjecture~\ref{conj:deligne} from curves to higher-dimensional varieties makes frequent use of a fact which we call the \emph{Lefschetz slicing principle}: for any closed point $x \in X$ and any irreducible coefficient object $\calE$ on $X$, after possibly making a finite base extension of $k$ (depending on $x$ and $\calE$), one can find a curve $C$ in $X$ containing $x$ such that the restriction of $\calE$ to $C$ is again irreducible. For $\ell \neq p$, this is shown by an analysis of the images of representations, as in \cite[\S 1.7]{deligne-finite} (correcting the proof of \cite[Proposition~VII.7]{lafforgue}; see also \cite[Appendix~B]{esnault-kerz}). Using cohomological arguments, we are able to make a fairly similar argument (Lemma~\ref{L:slicing}); a related but distinct approach has been taken by Abe--Esnault (see below).
\end{enumerate}

We mention briefly one further example for which cohomological considerations do not suffice.
For $\ell \neq p$, one can eliminate all wild monodromy at the boundary using a single finite \'etale cover
(by trivializing a residual representation); this plays an important role in the methods of Deligne and Drinfeld to attack aspects of Conjecture~\ref{conj:deligne} (see below). The analogous statement for $\ell = p$ is the 
\emph{semistable reduction theorem for overconvergent $F$-isocrystals}  (see Proposition~\ref{P:alter to tame}): for every $p$-adic coefficient object, one can pull back by a suitable alteration (in the sense of de Jong
\cite{dejong-alterations}) and then make a logarithmic extension to some good compactification.
The proof of this cannot be considered a ``cohomological argument'' because in fact the foundations of rigid cohomology are deeply rooted in this fact; see \cite[\S 11]{kedlaya-isocrystals} or the introduction to \cite{abe-companion}.

\subsection{Deligne's conjecture for $\dim(X) > 1$}

We now return to Conjecture~\ref{conj:deligne}.
By analogy with a conjecture of Simpson that quasi-unipotent rigid local systems on smooth complex varieties
should have a geometric/motivic origin (see \cite[Conjecture~4]{simpson} for the case of a proper variety),
it is generally expected that coefficient objects on a general $X$ should all arise from motives over $X$;
for any given coefficient object, such a description would imply Conjecture~\ref{conj:deligne} for this object.
(This conjecture is attributed to Deligne by Esnault--Kerz \cite[Conjecture~2.3]{esnault-kerz-notes}; see also
\cite[Question~1.4]{drinfeld-deligne}, \cite[Remark~1.5]{krishnamoorthy-pal}.)

Unfortunately, there seems to be no analogue of the Langlands correspondence which would provide geometric origins for coefficient objects on higher-dimensional varieties,
except in the case of weight 1 where some results are known (as in \cite{krishnamoorthy, krishnamoorthy-pal}),
and some isolated cases which can be described completely in terms of weight-1 data and hence might be tractable by similar methods (\textit{e.g.} weight-2 motives of K3 surfaces, via the Kuga--Satake construction). Instead, all known results towards Conjecture~\ref{conj:deligne} in cases where $\dim(X)>1$ proceed by using Theorem~\ref{T:one-dim case} as a black box, restricting to curves in $X$ to get a system of interrelated coefficient objects (such a system is called\footnote{As per \cite[\S 4]{cadoret}, this terminology is in fact due to Kindler.} a \emph{$2$-skeleton sheaf} by Esnault--Kerz).
\begin{enumerate}[label=$\bullet$]
\item
Statements (i), (iii), (iv) are pointwise conditions on $\calE$, so they immediately promote to general $X$ using the Lefschetz slicing principle.
\item
Statement (ii) was shown by Deligne \cite{deligne-finite} for $\ell \neq p$ using a numerical argument to give a uniform bound for the number field generated by the Frobenius traces over all curves in $X$.
The argument transposes to $\ell=p$, using semistable reduction for $p$-adic coefficient objects (see above) to provide the requisite control of wild ramification.
\item
Statement (v) was shown by Drinfeld \cite{drinfeld-deligne} for $\ell \neq p$ using Deligne's work on (ii) plus an argument to piece together companions on curves; the latter draws on the use of Hilbert irreducibility in the work of Wiesend \cite{wiesend1}. The argument transposes to $\ell=p$, again using semistable reduction for $p$-adic coefficient objects.
\end{enumerate}

To sum up, we have in hand the following statement; see Corollary~\ref{C:companion ell} for the proof.

\begin{theorem} \label{T:deligne conj2}
Parts (i)--(v) of Conjecture~\ref{conj:deligne} hold as written $($whether or not $\ell=p)$.
\end{theorem}

At this point, we gratefully acknowledge the recent joint paper of Abe--Esnault \cite{abe-esnault},
in which a result which is essentially equivalent to Theorem~\ref{T:deligne conj2} is proved.
That paper and this one constitute in some sense a virtual joint project, in that many of the ideas were discussed among the two sets of authors at an embryonic stage; we have opted to write the proofs up separately to illustrate different ways to assemble the main ingredients. See \S\ref{sec:abe-esnault} for a direct comparison.

We also wish to emphasize here that we are assuming that $X$ is smooth, rather than normal.
While parts (i)--(iv) are indifferent to this distinction (as discussed in detail in \cite{cadoret}),
parts (v) and (vi) are not; the examples of \cite[\S 6]{drinfeld-deligne} show that one cannot directly apply the method of skeleton sheaves when $X$ is not smooth, and so some additional ideas are needed.

\subsection{Towards crystalline companions}

Going forward, we are left to address part (vi) of Conjecture~\ref{conj:deligne},
or equivalently the following conjecture.
\begin{conj}  \label{conj:companion}
Any $\ell$-adic coefficient object on $X$ which is $E$-algebraic for some number field $E$
admits an $\ell'$-adic companion.
\end{conj}
In light of Theorems~\ref{T:one-dim case} and~\ref{T:deligne conj2},
only the case $\ell' = p$, $\dim(X) > 1$ remains open; it is moreover harmless to also assume $\ell \neq p$.
We resolve this conjecture in a sequel paper \cite{kedlaya-companions2}.

In the opposite direction, we point out that this paper itself has a ``companion'' in the form of the expository article \cite{kedlaya-isocrystals}; that paper presents the basic properties of $p$-adic cohomology that will be used herein. Our general approach to citations in this paper is to refer to \cite{kedlaya-isocrystals}
rather than original sources; moreover, some notation and terminology from \cite{kedlaya-isocrystals} will be used  with limited explanation.
For these reasons (among others), it is strongly recommended to read \cite{kedlaya-isocrystals} before proceeding with this paper.

\subsection*{Acknowledgments}
This article is a byproduct of an invitation to the geometric Langlands seminar at the University of Chicago during spring 2016. Thanks to Tomoyuki Abe, Emiliano Ambrosi, Marco D'Addezio, Chris Davis, Valentina Di Proietto, Vladimir Drinfeld, H\'el\`ene Esnault, Ambrus P\'al, Atsushi Shiho, and Nobuo Tsuzuki for feedback.

\section{Coefficient objects}

We begin by resetting some basic notation and terminology from the introduction, then gathering some basic facts about coefficient objects in Weil cohomologies.

\begin{notation}
Throughout this paper,
let $k$ be a finite field of characteristic $p>0$;
let $\ell$ denote an arbitrary prime \emph{different from $p$};
let $K$ denote the fraction field of the ring $W(k)$ of $p$-typical Witt vectors with coefficients in $k$;
and let $X$ denote a smooth scheme of finite type over $k$. 
Let $X^\circ$ denote the set of closed points of $X$.
For $x \in X^\circ$, let $\kappa(x)$ denote the residue field of $x$ and let $q(x)$ denote the cardinality of $\kappa(x)$.

By a \emph{curve} over $k$, we mean a scheme which is smooth of relative dimension 1 and geometrically irreducible, but not necessarily proper, over $k$.
A \emph{curve in $X$} is a locally closed subscheme of $X$ which is a curve over $k$.

For $n$ a positive integer, let $k_n$ be the unique extension of $k$ of degree $n$
within a fixed algebraic closure of $k$
and put $X_n := X \times_k k_n$.
\end{notation}

\subsection{\'Etale and crystalline companions}
\label{sec:companions}

We first introduce the relevant Weil cohomology theories and their coefficient objects. For crystalline coefficients, see \cite{kedlaya-isocrystals} for a more thorough background discussion.

\begin{notation}
For $L$ an algebraic extension of $\QQ_\ell$, let $\Weil(X)$ (resp.\ $\Weil(X) \otimes L$) denote the category of lisse Weil $\QQ_\ell$-sheaves (resp.\ lisse Weil $L$-sheaves) on $X$. The concrete definition of this category is not immediately relevant, so we postpone recalling it until Definition~\ref{D:lisse weil}.

Let $\FIsoc^\dagger(X)$
denote the category of overconvergent $F$-isocrystals over $X$.
For $L$ an algebraic extension of $\QQ_p$, let $\FIsoc^\dagger(X) \otimes L$
denote the category of overconvergent $F$-isocrystals over $X$ with coefficients in $L$, in the sense of \cite[Definition~9.2]{kedlaya-isocrystals}.

By a \emph{coefficient object} on $X$, we will mean an object of one of the categories $\Weil(X) \otimes \overline{\QQ}_\ell$ or $\FIsoc^\dagger(X) \otimes \overline{\QQ}_p$. We will use this definition frequently to make uniform statements; when it becomes necessary to separate these statements into their two constituents, we will distinguish these as, respectively, the \emph{\'etale case} and the \emph{crystalline case}. We also refer to
$\ell$ (resp.\ $p$) as the \emph{coefficient prime}, to $\QQ_\ell$ (resp. $\QQ_p$) as the \emph{base coefficient field}, and to $\overline{\QQ}_\ell$ (resp.\ $\overline{\QQ}_p$) as the \emph{full coefficient field}.
Note that the definition of a coefficient object is functorial in $X$ alone, without its structure morphism to $\Spec(k)$; this means we can generally assume that $X$ is geometrically irreducible in what follows.

We say that a coefficient object on $X$ is \emph{constant} if it is the pullback of a coefficient object on $\Spec(k)$. This definition also does not depend on $k$. The operation of tensoring with a constant object of rank $1$ will be referred to as a \emph{constant twist}.
\end{notation}

\begin{lemma} \label{L:restriction functor}
Let $U$ be an open dense subscheme of $X$. Then the restriction functors from coefficient objects on $X$ to coefficient objects on $U$ 
are fully faithful and preserve irreducibility; moreover, they are equivalences whenever
$X\setminus U$ has codimension at least $2$ in each component of $X$.
\end{lemma}
\begin{proof}
In the \'etale case, these assertions are straightforward consequences of the interpretation of lisse sheaves (when $X$ is connected) in terms of the geometric \'etale fundamental group,
plus Zariski--Nagata purity \cite[Tag~0BMB]{stacks-project} in the case of the final assertion.
In the crystalline case, more difficult arguments is required; see \cite[Theorems~5.1, 5.3, 5.11]{kedlaya-isocrystals} for attributions.
\end{proof}

\begin{lemma} \label{L:decompose to finite order}
Assume that $X$ is connected.
For every coefficient object $\calE$ on $X$,
there exists a constant twist of $\calE$ whose determinant is of finite order.
\end{lemma}
\begin{proof}
This reduces easily to the case where $\calE$ is itself of rank 1.
For this, see \cite[1.2]{deligne-finite} in the \'etale case and 
\cite[Lemma~9.14]{kedlaya-isocrystals} in the crystalline case.
\end{proof}

\begin{defn}
For $\calE$ a coefficient object on $X$, let $H^i(X, \calE)$ denote the cohomology groups of $\calE$ (without supports); these are finite-dimensional vector spaces over the full coefficient field (see \cite[Theorem~8.4]{kedlaya-isocrystals} in the crystalline case).
\end{defn}

\begin{defn} \label{D:Frobenius at points}
For $\calE$ a coefficient object on $X$ and $x \in X^\circ$, let $\calE_x$ be the pullback of $\calE$ to $x$;
it admits a linear action of the $q(x)$-power geometric Frobenius,
whose reverse characteristic polynomial we denote by $P(\calE_x, T)$.
(See \cite[Definition~9.5]{kedlaya-isocrystals} for discussion of the crystalline case.)

We say that $\calE$ is \emph{algebraic} if for all $x \in X^\circ$, $P(\calE_x, T)$ has coefficients in the field of algebraic numbers. We say that $\calE$ is \emph{uniformly algebraic} if these coefficients all belong to a single number field, or \emph{$E$-algebraic} if this field can be taken to be $E$.
For example, if $\calE$ is of rank 1 and of finite order, then the roots of the polynomials $P(\calE_x, T)$ are all roots of unity of bounded order, so $\calE$ is uniformly algebraic.

Suppose that $\calE$ is algebraic and fix a finite place $\lambda$ of the algebraic closure of $\QQ$ in the full coefficient field. For $x \in X^\circ$, we define the \emph{normalized Newton polygon}
of $\calE$ at $x$ with respect to $\lambda$, denoted $N_x(\calE, \lambda)$, to be the boundary of the
lower convex hull of the set of points
\[
\left\{ \left( i, \frac{1}{[\kappa(x)\colon\! k]} v_\lambda(a_{i}) \right)\mid i = 0,\dots,d \right\} \subset \mathbb{R}^2.
\]
\end{defn}

\begin{defn}
Let $\calE$ and $\calF$ be two coefficient objects, possibly in different categories (and possibly for different values of $\ell$),
and fix an identification $\iota$ of the algebraic closures of $\QQ$ within the respective full coefficient fields.
We say that $\calE$ and $\calF$ are \emph{companions} (with respect to $\iota$) if for each $x \in X^\circ$, the polynomials $P(\calE_x, T)$ and $P(\calF_x, T)$ are equal to the same element of $\overline{\QQ}[T]$; in particular, this implies that $\calE$ and $\calF$ are both algebraic on each connected component of $X$.
We will see later that within their respective categories, $\calE$ and $\calF$ uniquely determine each other up to
semisimplification (Theorem~\ref{T:chebotarev}).
In case $\calF$ has coefficient prime $\ell$ (resp.\ $p$), we also call it an \emph{\'etale companion}  (resp.\ a \emph{crystalline companion}) of $\calE$ (again with respect to $\iota$).
\end{defn}

\begin{defn} \label{D:L-function}
For $\calE$ a coefficient object with coefficient prime $*$, let 
\[
L(\calE, T) := \prod_{x \in X^\circ} P(\calE_x, T^{[\kappa(x)\colon\! k]})^{-1} \in \overline{\QQ}_* \llbracket T \rrbracket
\]
be its associated $L$-function. 
The Lefschetz trace formula for Frobenius (see \cite[Theorem~9.6]{kedlaya-isocrystals} in the crystalline case) asserts that
\begin{equation} \label{eq:lefschetz}
L(\calE, T) = \prod_{i=0}^{2d} \det(1 - (\#k)^{-d} F^{-1} T, H^i(X, \calE^\dual))^{(-1)^{i+1}}.
\end{equation}
(Note the dualization here, caused by working with cohomology without supports.) Consequently,
the Euler characteristic 
\[
\chi(\calE) := \sum_{i=0}^{2\dim(X)} (-1)^i \dim_{\overline{\QQ}_*} H^i(X, \calE)
\]
of $\calE$ equals the order of vanishing of $L(\calE^\dual,T)$ at $T=\infty$.
In particular, any two companions have the same $L$-function and hence the same Euler characteristic.
\end{defn}

\subsection{Slopes of isocrystals}
\label{sec:Newton poly}

We next review some crucial properties of Newton polygons in the crystalline case. For this, we must distinguish between convergent and overconvergent $F$-isocrystals; again, see \cite{kedlaya-isocrystals} for a more detailed discussion of this fundamental dichotomy.

\begin{defn}
For $L$ an algebraic extension of $\QQ_p$,
let $\FIsoc(X)$ (resp.\ $\FIsoc(X) \otimes L$) 
denote the category of convergent $F$-isocrystals on $X$ (resp. convergent $F$-isocrystals on $X$
with coefficients in $L$).
There is a restriction functor $\FIsoc^\dagger(X) \otimes L \to \FIsoc(X) \otimes L$ which is an equivalence of categories when $X$ is proper (so in particular when $X = \Spec(k)$), and
fully faithful in general (see below).
We may thus extend definitions pertaining to convergent $F$-isocrystals, such as Newton polygons,
to overconvergent $F$-isocrystals via restriction.
\end{defn}

\begin{lemma} \label{L:restriction convergent}
The restriction functor $\FIsoc^\dagger(X) \otimes L \to \FIsoc(X) \otimes L$ is fully faithful.
Moreover, for $U$ an open dense subscheme of $X$, the restriction functor $\FIsoc(X) \otimes L \to \FIsoc(U) \otimes L$
is fully faithful.
\end{lemma}
\begin{proof}
See \cite[Theorem~5.3]{kedlaya-isocrystals}.
\end{proof}

\begin{defn} \label{D:Newton polygon}
For $\calE \in \FIsoc(X)$, we define the \emph{Newton polygon} $N_x(\calE)$ for $x \in X$ (not necessarily a closed point) using the Dieudonn\'e--Manin classification theorem, as in \cite[\S 3]{kedlaya-isocrystals}. We extend the definition to $\calE \in \FIsoc(X) \otimes \overline{\QQ}_p$ by renormalization. By convention, $N_x(\calE)$ is convex with left endpoint at $(0,0)$ and right endpoint at $(\rank_{\overline{\QQ}_p}(\calE), *)$; its vertices belong to $\ZZ \times [L:\QQ_p]^{-1} \ZZ$ where $L$ is a finite extension of $\QQ_p$ for which $\calE \in \FIsoc(X) \otimes L$.
We say $\calE$ is \emph{unit-root} (or \emph{\'etale}) if for all $x \in X$, $N_x(\calE)$ has all slopes equal to 0. 
\end{defn}

\begin{lemma} \label{L:same NP}
For $\calE \in \FIsoc(X) \otimes \overline{\QQ}_p$ and
$x \in X^\circ$, write $P(\calE_x, T) = \sum_{i=0}^d a_i T^i \in \overline{\QQ}_p[T]$ with $a_0 = 1$.
Then $N_x(\calE)$ equals the boundary of the
lower convex hull of the set of points
\[
\left\{ \left( i, \frac{1}{[\kappa(x)\colon\! k]} v_p(a_{i}) \right)\mid i = 0,\dots,d \right\} \subset \mathbb{R}^2.
\]
In particular, if $\calE\in \FIsoc(X)^\dagger \otimes \overline{\QQ}_p$, then $N_x(\calE) = N_x(\calE, \lambda)$ where $\lambda$ is the place defined by the embedding $\overline{\QQ} \hookrightarrow \overline{\QQ}_p$.
\end{lemma}
\begin{proof}
We may assume at once that $X = \{x\}$.
Using the Dieudonn\'e--Manin decomposition, we may further reduce to the case where $\calE$ is unit-root;
we may then also assume that $L = \QQ_p$.
In this case, an object of $\FIsoc(X)$
is a finite-dimensional $K$-vector space $V$ equipped with an isomorphism $F: \sigma^* V \to V$,
and the unit-root condition corresponds to the existence of a $W(k)$-lattice $T$ in $V$ such that $F$ induces an isomorphism $\sigma^* T \cong T$. Using such a lattice to compute $P(\calE_x,T)$, we see that the latter has Newton polygon with all slopes equal to 0; this proves the claim.
\end{proof}

\begin{prop} \label{P:unit-root}
Assume that $X$ is connected and fix a geometric point $\overline{x}$ of $X$.
For each finite extension $L$ of $\QQ_p$, the following statements hold.
\begin{enumerate}
\item[(a)]
The category of unit-root objects in $\FIsoc(X) \otimes L$ is equivalent to the category of
continuous representations of $\pi_1(X, \overline{x})$ on finite-dimensional $L$-vector spaces.
\item[(b)]
Under the equivalence in (a), the unit-root objects in $\FIsoc^\dagger(X) \otimes L$ correspond to representations which are \emph{potentially unramified}. This condition means that after passing from $X$ to some connected finite \'etale cover $X'$ through which $\overline{x} \to X$ factors, the restriction from $\pi_1(X', \overline{x})$ to the inertia group of any divisorial valuation is the trivial representation.
\end{enumerate}
\end{prop}
\begin{proof}
See \cite[Theorem~9.4]{kedlaya-isocrystals}.
\end{proof}

\begin{cor}\label{C:finite order companion}
Every coefficient object on $X$ which is algebraic of rank $1$ admits companions in all categories.
\end{cor}
\begin{proof}
We are free to check the claim after a constant twist, so by Lemma~\ref{L:decompose to finite order} we may ensure that $\calE$ is of finite order. Fix a geometric point $\overline{x}$ of $X$. By definition (in the \'etale case) or Proposition~\ref{P:unit-root} (in the crystalline case), $\calE$ corresponds to a continuous homomorphism $\rho: \pi_1(X, \overline{x}) \to \mu_{\infty}$ for the discrete topology on the target; we may then reverse the logic in any other category.
\end{proof}

\begin{remark} \label{R:unit-root truncated}
Proposition~\ref{P:unit-root} also has closely related integral and truncated versions. For instance,
the category of finite projective $\calO_X$-modules equipped with isomorphisms with their $\varphi_k$-pullbacks is isomorphic to the category of continuous representations of $\pi_1(X)$ on finite-dimensional $k$-vector spaces \cite[Proposition 1.1]{sga7}.
\end{remark}

\begin{prop} \label{P:specialization}
For $\calE \in \FIsoc(X) \otimes \overline{\QQ}_p$, the function $x \mapsto N_x(\calE)$ is upper semicontinuous
and the right endpoint of $N_x(\calE)$ is locally constant on $X$.
In particular, if $X$ is a curve over $k$ with generic point $\eta$, then there are only finitely many $x \in X^\circ$ for which $N_x(\calE) \neq N_\eta(\calE)$.
\end{prop}
\begin{proof}
This reduces immediately to the case $\calE \in \FIsoc(X)$, for which see \cite[Theorem~3.12]{kedlaya-isocrystals}.
\end{proof}

\begin{prop} \label{P:slope filtration}
Choose $\calE \in \FIsoc(X) \otimes L$ for some finite extension $L$ of $\QQ_p$.
\begin{enumerate}
\item[(a)]
Suppose the point $(m,n) \in \ZZ \times \QQ$ is a vertex of $N_x(\calE)$ for each $x \in X$.
Then there exists a short exact sequence
\[
0 \to \calE_1 \to \calE \to \calE_2 \to 0
\]
in $\FIsoc(X) \otimes L$ with $\rank \calE_1 = m$ such that for each $x \in X$, $N_x(\calE_1)$
is the portion of $N_x(\calE)$ from $(0,0)$ to $(m,n)$.

\item[(b)]
Suppose that the function $x \mapsto N_x(\calE)$ is constant. Then there exists a unique filtration
\[
0 = \calE_0 \subset \cdots \subset \calE_l = \calE
\]
in $\FIsoc(X) \otimes L$ with the property that for some sequence $s_1 < \cdots < s_l$, for each $i$ the quotient $\calE_i/\calE_{i-1}$ has constant Newton polygon with all slopes equal to $s_i$.
This filtration is called the \emph{slope filtration} of $\calE$.
\end{enumerate}
\end{prop}
\begin{proof}
See \cite[Theorem~4.1, Corollary~4.2]{kedlaya-isocrystals}.
\end{proof}

\begin{prop}[Drinfeld--Kedlaya, Kramer-Miller] \label{P:uniform bound on slopes}
For $X$ irreducible with generic point $\eta$ 
and $\calE \in \FIsoc^\dagger(X) \otimes \overline{\QQ}_p$ indecomposable, no two consecutive slopes of $N_\eta(\calE)$ differ by more than $1$.
\end{prop}
\begin{proof}
By Lemma~\ref{L:restriction convergent} this reduces to the corresponding assertion in $\FIsoc(X) \otimes \overline{\QQ}_p$, for which see \cite[Theorem~1.1.5]{drinfeld-kedlaya}. An alternate proof has been communicated to the author by 
Kramer-Miller; it will appear elsewhere.
\end{proof}

\subsection{Monodromy groups}
\label{sec:mono}

We next introduce the formalism of monodromy groups,
following \cite[\S 9]{kedlaya-isocrystals} and \cite[\S 3.2]{daddezio} in the crystalline case. (The original reference in the crystalline case is \cite[\S 5]{crew-mono}, but it has somewhat restrictive hypotheses which are not suitable for our present discussion; see
\cite[Remark~9.9]{kedlaya-isocrystals}.)
\begin{hypothesis} \label{H:mono}
Throughout \S\ref{sec:mono}, 
assume that $X$ is connected, and fix a closed point $x \in X^\circ$ and a geometric point $\overline{x} = \Spec(\overline{k})$ of $X$ lying over $x$.
In the crystalline case, also choose an embedding $W(\overline{k}) \hookrightarrow \overline{\QQ}_p$.
\end{hypothesis}

\begin{defn} \label{D:mono}
On any category of coefficient objects, let $\omega_x$ be the neutral (over the full coefficient field) fiber functor defined by pullback from $X$ to $x$.
Note that in this construction, we implicitly discard the Frobenius structure: coefficient objects on $x$ correspond to vector spaces equipped with a fixed automorphism (see Definition~\ref{D:Frobenius at points}), 
and $\omega_x$ does not retain the data of the automorphism. (In the crystalline case,
this definition depends on the choice of the embedding  $W(\overline{k}) \hookrightarrow \overline{\QQ}_p$;
see \cite[Definition~9.5]{kedlaya-isocrystals}.)

For $\calE$ a coefficient object on $X$, let $G(\calE)$ denote the automorphism group of $\omega_x$ on the Tannakian category generated by $\calE$. 
The group $G(\calE)$ is the \emph{arithmetic monodromy group} associated to $\calE$.
This group turns out not to be especially useful for our purposes; compare with Definition~\ref{D:geom mono}.
\end{defn}

In the \'etale case, the construction of Definition~\ref{D:mono} has a rather concrete interpretation which we recall here.

\begin{defn} \label{D:lisse weil}
Suppose that $X$ is geometrically irreducible (by enlarging $k$ if necessary), and consider the exact sequence
\[
1 \to \pi_1^{\mathrm{et}}(X_{\overline{k}}, \overline{x}) \to \pi_1^{\mathrm{et}}(X, \overline{x}) \to 
G_k \to 1.
\]
The group $G_k$ is isomorphic to $\widehat{\ZZ}$; let us normalize this isomorphism so that geometric Frobenius corresponds to 1. We may then pull back from $\widehat{\ZZ}$ to $\ZZ$ to obtain an exact sequence
\[
1 \to \pi_1^{\mathrm{et}}(X_{\overline{k}}, \overline{x}) \to W_X \to 
\ZZ \to 1;
\]
the group $W_X$ is the \emph{Weil group} associated to $X$, and for any algebraic extension $L$ of $\QQ_\ell$,
the category $\Weil(X) \otimes L$ may be identified
with (or even defined as) the category of continuous representations of $W_X$ on finite-dimensional $L$-vector spaces.
(To use this as the definition when $L$ is not finite over $\QQ_\ell$, one must check that the definition commutes with colimits; for this, see \cite[Remark~9.0.7]{katz-sarnak}.)
 For any object $\calE$ in this category, $G(\calE)$ may then be naturally identified with the Zariski closure of the image of $W_X$ in the associated representation.
\end{defn}

\begin{defn} \label{D:geom mono}
Retain notation as in Definition~\ref{D:mono}. In the \'etale case, we define the \emph{geometric monodromy group}
$\overline{G}(\calE)$ to be the Zariski closure of the image of $\pi_1^{\mathrm{et}}(X_{\overline{k}}, \overline{x})$ in the associated representation. Equivalently, define the category of \emph{geometric coefficient objects} by forming the category of continuous representations of $\pi_1^{\mathrm{et}}(X_{\overline{k}}, \overline{x})$ on finite dimensional $\overline{\QQ}_\ell$-vector spaces, then extracting the Tannakian subcategory generated by coefficient objects.
We may then equivalently define $\overline{G}(\calE)$ as the automorphism group of $\omega_x$ on the Tannakian subcategory of geometric coefficient objects generated by $\calE$.

In the crystalline case, we define the category of \emph{geometric coefficient objects} by forming the category of 
overconvergent isocrystals (without Frobenius structure) with coefficients in $\overline{\QQ}_p$ (by a suitable 2-colimit over finite extensions of $\QQ_p$, as described in \cite[Definition~9.2]{kedlaya-isocrystals} in the presence of Frobenius structures) and then extracting the Tannakian subcategory generated by coefficient objects. We then define $\overline{G}(\calE)$ again as the automorphism group of $\omega_x$ on the Tannakian subcategory of geometric coefficient objects generated by $\calE$. We may then turn around and define the 
\emph{Weil group} of $\calE$ as the semidirect product of $\overline{G}(\calE)$ by $\ZZ$ via the action of Frobenius.
\end{defn}
 
\begin{remark} \label{R:crew mono}
The definition of monodromy groups given here is essentially \cite[Definition~3.2.4]{daddezio};
in the crystalline case, see also \cite[Definition~9.8]{kedlaya-isocrystals} and subsequent remarks.
In particular, as pointed out in \cite[Remark~3.2.5]{daddezio}, the definition of $\overline{G}(\calE)$ agrees with Crew's definition of the differential Galois group \cite[\S 2]{crew-mono} when $X$ is a curve and $x \in X(k)$.
\end{remark}

\begin{remark} \label{R:open dense}
For any open dense subset $U$ of $X$ containing $x$, the inclusion $\overline{G}(\calE|_U) \to \overline{G}(\calE)$ is an isomorphism. This is straightforward in the \'etale case; in the crystalline case, see \cite[Remark~9.10]{kedlaya-isocrystals}. (See also Lemma~\ref{L:ae same mono} below.)
\end{remark}

\begin{remark} \label{R:finite monodromy}
By definition, any coefficient object $\calE$ on $X$ belongs to $\Weil(X) \otimes L$ or $\FIsoc^\dagger(X) \otimes L$ for some finite extension $L$ of the base coefficient field. Choose such an $L$, then suppose that $\overline{G}(\calE)$ is finite.
In the \'etale case, it is obvious that $\calE$ gives rise to a discrete $L$-representation of $\pi_1(X_{\overline{k}}, \overline{x})$. It is a key observation of Crew that this also holds in the crystalline case;
namely, the underlying isocrystal of $\calE$ admits a canonical unit-root Frobenius structure (see \cite[Corollary~9.17]{kedlaya-isocrystals}) and then Proposition~\ref{P:unit-root} applies
(see \cite[Lemma~9.18]{kedlaya-isocrystals}).
\end{remark}

\begin{remark} \label{R:constant subobject}
As a special case of Remark~\ref{R:finite monodromy},
note that a coefficient object $\calE$ on $X$ is constant if and only if $\overline{G}(\calE)$ is trivial. In particular, this condition passes to subobjects.
\end{remark}

\begin{remark} \label{R:absolutely irreducible}
As the terminology suggests, the definition of the geometric monodromy group is invariant under base extension on $k$.
However, it is not the case that the geometric monodromy group coincides with the arithmetic monodromy group after a sufficiently large base extension. For example, if $\calE$ is constant of rank $1$ but not of finite order, then $G(\calE_{k'}) = \mathbb{G}_m$ for all finite extensions $k'$ of $k$, whereas $\overline{G}(\calE)$ is trivial.

One consequence of this is that it is not immediately obvious whether a coefficient object which is \emph{absolutely irreducible} (by which we mean it remains irreducible after any finite extension of the base field) is
\emph{geometrically irreducible} (by which we mean it is irreducible as a geometric coefficient object,
or equivalently that $\overline{G}(\calE)$ acts irreducibly on its tautological representation).
This does turn out to be true; 
moreover, for any irreducible  coefficient object $\calE$ on $X$,
there exists a positive integer $n$ such that on $X_n$,
$\calE$ splits as a direct sum of geometrically irreducible subobjects 
$\calE_1 \oplus \cdots \oplus \calE_n$ over $X_n$ which are cyclically permuted by the action of the $k$-Frobenius.
See \cite[Remark~9.13]{kedlaya-isocrystals} and references therein.
\end{remark}

In light of Remark~\ref{R:absolutely irreducible}, we may refine Lemma~\ref{L:decompose to finite order} using the following statement.

\begin{lemma} \label{L:decompose to finite order2}
Suppose that $\calE$ is irreducible and $\det(\calE)$ is of finite order. Then for any positive integer $n$, every irreducible constituent of the pullback of $\calE$ to $X_n$ has determinant of finite order.
\end{lemma}
\begin{proof}
By Remark~\ref{R:absolutely irreducible},
$\calE$ splits as a direct sum of irreducible subobjects 
$\calE_1 \oplus \cdots \oplus \calE_m$ over $X_n$ which are cyclically permuted by the action of the $k$-Frobenius.
By Lemma~\ref{L:decompose to finite order}, we can choose a coefficient object $\calL'$ on $k_n$ of rank $1$
such that $\calE_1 \otimes \calL'$ has determinant of finite order. By extracting a suitable $n$-th root,
we may realize $\calL$ (not uniquely) as the pullback of a coefficient object $\calL$ on $k$ of rank 1.
The products $\calE_1 \otimes \calL, \dots, \calE_n \otimes \calL$ are then cyclically permuted by the action of the $k$-Frobenius; consequently, since the first one has finite order, so do the others. In particular, $\det(\calE \otimes \calL)$ is of finite order, as then is $\calL$, as then are $\det(\calE_1),\dots,\det(\calE_n)$.
\end{proof}

\begin{prop} \label{P:component group}
For any coefficient object $\calE$ on $X$,
there is a natural surjective continuous homomorphism $\psi_{\calE}: \pi_1^{\mathrm{et}}(X, \overline{x}) \to \pi_0(G(\calE))$,
which induces a surjective continuous homomorphism
$\overline{\psi}_{\calE}: \pi_1^{\mathrm{et}}(X_{\overline{k}}, \overline{x}) \to \pi_0(\overline{G}(\calE))$.
\end{prop}
\begin{proof}
In the \'etale case, this is apparent from Definition~\ref{D:lisse weil}.
For the crystalline case, see \cite[Proposition~3.3.4]{daddezio}.
\end{proof}

The following is a highly nontrivial result of Grothendieck in the \'etale case and Crew in the crystalline case.
See \cite{daddezio} for numerous consequences of this statement.
\begin{prop} \label{P:global monodromy}
For any coefficient object $\calE$ on $X$, the radical of $\overline{G}(\calE)^{\circ}$ is unipotent. Consequently,
if $\calE$ is semisimple, then $\overline{G}(\calE)^\circ$ is semisimple.
\end{prop}
\begin{proof}
See \cite[Theorem~3.4.5]{daddezio} for the first statement
(or \cite[Theorem~9.19]{kedlaya-isocrystals} for the crystalline case). The passage from the first statement to the second is purely group-theoretic; see \cite[Corollary~4.10]{crew-mono}.
\end{proof}

\begin{remark}
One can also consider geometric monodromy groups of convergent $F$-isocrystals; in fact, the relationship
between the geometric monodromy group of an overconvergent $F$-isocrystal and that of its underlying
convergent $F$-isocrystal provides much useful information. However, we will not make explicit use of this here.
See \cite{daddezio} for a modern treatment.
\end{remark}

\subsection{Tame and docile coefficients}

In much of the study of companions, a crucial role is played by coefficient objects for 
which the local monodromy at the boundary is tamely ramified and unipotent.

\begin{defn} \label{D:tame log}
Let $X \hookrightarrow \overline{X}$ be an open immersion with dense image.
Let $D$ be an irreducible divisor of $\overline{X}$ with generic point $\eta$.
Let $\calE$ be a coefficient object on $X$. 
In the \'etale case, we say that $\calE$ is \emph{tame} (resp.\ \emph{docile}) along $D$ if the action of the inertia group of $\eta$ on $\calE$ (arising from Definition~\ref{D:lisse weil}) is tamely ramified (resp.\ tamely ramified and unipotent).

In the crystalline case, we say that $\calE$ is \emph{tame} (resp.\ \emph{docile}) along $D$ if $\calE$ has
$\QQ$-unipotent monodromy in the sense of \cite[Definition~1.3]{shiho-log}
(resp. unipotent monodromy in the sense of \cite[Definition~4.4.2]{kedlaya-semi1}) along $D$.

We say that $\calE$ is \emph{tame} (resp.\ \emph{docile}) if it is so with respect to every choice of $\overline{X}$ and $D$. By definition, this property is stable under arbitrary pullbacks.
\end{defn}

\begin{remark} \label{R:tame in codim 1}
The category of tame (resp.\ docile) coefficient objects on $X$ is stable under subquotients and extensions. This is clear in the \'etale case; in the crystalline case, see \cite[Proposition~3.2.20, Proposition~6.5.1]{kedlaya-semi1}.
\end{remark}

To make this definition more useful in practice, we need the concept of a \emph{good compactification}.
\begin{defn}
A \emph{smooth pair} over $k$ is a pair $(Y,Z)$ in which $Y$ is a smooth $k$-scheme and $Z$ is a strict normal crossings 
divisor on $Y$; we refer to $Z$ as the \emph{boundary} of the pair. (Note that $Z = \emptyset$ is allowed.)
A \emph{good compactification} of $X$ is a smooth pair $(\overline{X}, Z)$ over $k$ with $\overline{X}$ projective (not just proper) over $k$, together with an isomorphism $X \cong \overline{X} \setminus Z$; we will generally treat the latter as an identification.

If $\overline{X}$ is a good compactification of $X$ with boundary $Z$, then a coefficient object on $X$ is tame (resp.\ docile) if and only if it is so with respect to each component of $Z$. Namely, this follows from Zariski--Nagata purity in the \'etale case, and from \cite[Theorem~6.4.5]{kedlaya-semi1} in the crystalline case.
\end{defn}

For general $X$, the existence of a good compactification is (to the best of our understanding) not known; it would follow from resolution of singularities in characteristic $p$. As a workaround, we use de Jong's theorem on alterations.

\begin{defn}
An \emph{alteration} of $X$ is a morphism $f: X' \to X$ which is proper, surjective, and generically finite \'etale. This corresponds to a \emph{separable alteration} in the sense of de Jong \cite{dejong-alterations}.
\end{defn}

\begin{prop}[de Jong] \label{P:alterations}
There exists an alteration $f: X' \to X$ such that $X'$ admits a good compactification. (Beware that $X'$ is not guaranteed to be geometrically irreducible over $k$.)
\end{prop}
\begin{proof}
Keeping in mind that $k$ is perfect, see
\cite[Theorem~4.1]{dejong-alterations}.
\end{proof}

We will frequently make use of the fact that the tame and docile conditions can always be achieved after pullback along a suitable alteration.

\begin{prop} \label{P:alter to tame}
For any coefficient object $\calE$ on $X$, there exists an alteration $f: X' \to X$
such that $X'$ admits a good compactification and $f^* \calE$ is docile.
$($As in Proposition~\ref{P:alterations}, we cannot guarantee that $X'$ is geometrically irreducible over $k.)$
\end{prop}
\begin{proof}
We first treat the \'etale case.
We may assume that $X$ is connected and choose a geometric basepoint $\overline{x}$.
By Definition~\ref{D:lisse weil}, $\calE$ corresponds to a representation of $W_X$ on some finite-dimensional $L$-vector space $V$. The action of $\pi_1^{\mathrm{et}}(X_{\overline{k}}, \overline{x})$ on $V$ has compact image in
$\GL(V)$, so we may choose a lattice $T$ stable under this action. By taking $f$ to trivialize the mod-$\ell$ action on $T$ (using Proposition~\ref{P:alterations}), we enforce that the actions of inertia are all tame. They are also quasi-unipotent by the usual argument of Grothendieck: the eigenvalues of Frobenius form a multiset of length at most $r := \rank(\calE)$ which is stable under taking $p$-th powers, so this multiset must consist entirely of roots of unity. To upgrade from quasi-unipotence to unipotence, it suffices to further trivialize the action on $T$ modulo some power of $\ell$, which may be bounded either in terms of $p$ and $r$ (by the previous consideration)
or $k$ (because $k$ only contains finitely many $\ell$-power roots of unity).

In the crystalline case, we instead apply the semistable reduction theorem for overconvergent $F$-isocrystals \cite[Theorem~7.6]{kedlaya-isocrystals}.
\end{proof}

\begin{remark} \label{R:tamely} 
In the \'etale case, the proof of Proposition~\ref{P:alter to tame} also shows that there exists a finite \'etale cover $f_0: X_0' \to X$ such that $f^* \calE$ is docile (but we cannot then guarantee the existence of a good compactification without a further alteration). 
After establishing the existence of \'etale companions, this will also hold in the crystalline case; see
Remark~\ref{R:semistable reduction}.
\end{remark}

\begin{remark} \label{R:tame to docile}
In the crystalline case, our notion of a tame coefficient object matches that of Abe--Esnault
\cite[\S 1.2]{abe-esnault}. Note that in either the \'etale or crystalline case,
a coefficient object on $X$ is tame if and only if it becomes docile after pullback along some finite \'etale cover of $X$ which is tamely ramified at the boundary.
\end{remark}

\begin{lemma} \label{L:extend tame}
A coefficient object on $X$ is tame (resp.\ docile) if and only if its restriction to every curve in $X$ is tame (resp. docile).
\end{lemma}
\begin{proof}
In the \'etale case, this is a straightforward consequence of Zariski--Nagata purity. In the crystalline case, a result of Shiho \cite[Theorem~7.4]{kedlaya-isocrystals} implies the docile version of the desired result; the tame version then follows using Remark~\ref{R:tame to docile}.
\end{proof}

\section{Curves}
\label{sec:dim 1}

We now introduce the use of the Langlands correspondence to prove Theorem~\ref{T:one-dim case}, thus resolving Conjecture~\ref{conj:deligne} when $\dim(X) = 1$; this amounts to a recitation of the works of L. Lafforgue and Abe.
We then make some additional calculations on curves in preparation for the study of higher-dimensional varieties.

\begin{hypothesis} \label{H:curve}
Throughout \S\ref{sec:dim 1}, assume that $X$ is a curve over $k$.
Let $\overline{X}$ be the unique smooth compactification of $X$ and 
put $Z := \overline{X} \setminus X$.
\end{hypothesis}

\subsection{The Langlands correspondence for curves}

We begin this discussion by recalling the statement of the Langlands correspondence for curves.
For compatibility with the literature in the \'etale case, we must introduce an additional definition.
\begin{defn} \label{D:lisse etale}
Suppose that $X$ is connected (but not necessarily a curve, for this definition only)
and let $\overline{x}$ be a geometric point of $X$.
For $L$ a finite extension of $\QQ_\ell$, a \emph{lisse \'etale $L$-sheaf} is a continuous representation of $\pi_1^{\mathrm{et}}(X, \overline{x})$ on a finite-dimensional $L$-vector space. By taking the 2-colimit over $L$, we obtain the category of \emph{lisse \'etale $\overline{\QQ}_\ell$-sheaves}. (We may then extend this definition to disconnected $X$ in an obvious fashion.)

Note that a continuous representation of $W_X$ extends to $\pi_1^{\mathrm{et}}(X, \overline{x})$ in at most one way (because $\ZZ$ is dense in $G_k$), and does so at all if and only if
its image has compact closure; from this observation, it follows that the lisse \'etale $L$-sheaves constitute a subcategory of $\Weil(X) \otimes L$ which is closed under extensions and subquotients.
By a similar group-theoretic argument \cite[Proposition~1.3.14]{deligne-weil2}, an irreducible lisse Weil $L$-sheaf is a lisse \'etale sheaf if and only if its determinant is a lisse \'etale sheaf; in particular, this is the case when the determinant is of finite order. See also Lemma~\ref{L:lisse valuation}.
\end{defn}

\begin{remark} \label{R:local Euler factors}
In the statement of Theorem~\ref{T:abe correspondence}, we refer to \emph{local Euler factors}
and \emph{local $\epsilon$-factors} associated to coefficient objects;
these are associated to a coefficient object $\calE$ on $X$ and a closed point $x \in \overline{X}$.
The local Euler factor is defined by pushing coefficient objects into a larger (derived) category of constructible objects and then pulling back to $x$; see \cite[\S A.3]{abe-companion} for a detailed description.
The construction of local $\epsilon$-factors
is due to Laumon \cite{laumon} in the \'etale case (extending work of Langlands and Deligne)
and Abe--Marmora \cite{abe-marmora}
in the crystalline case.

We will also need a more explicit description of the local Euler factor at $x$ in the case where $\calE$ is a crystalline coefficient object which is docile at $x$: it is the reverse characteristic polynomial of $F_x$ on
the kernel of the residue map at $x$. See for example \cite[Th\'eor\`eme~3.4.1]{caro-curve}.
\end{remark}

\begin{theorem}[L. Lafforgue, Abe] \label{T:abe correspondence}
Fix an embedding of $\overline{\QQ}$ into the full coefficient field.
Then for each positive integer $r$,
there is a bijection between irreducible coefficient objects of rank $r$ on $X$ with determinant of finite order,
and cuspidal automorphic representations of $\GL_r(\AAA_{k(X)})$ with central character of finite order which are unramified in $X$.
Moreover, this bijection may be chosen so that
Frobenius eigenvalues at points of $X$ correspond to Hecke eigenvalues at unramified places,
while local Euler factors and $\epsilon$-factors at points of $Z$ correspond
to the analogous quantities at ramified places.
\end{theorem}
\begin{proof}
See \cite{lafforgue} for the \'etale case and \cite[Theorem~4.2.2]{abe-companion} for the crystalline case.
\end{proof}

\begin{remark}
As noted earlier, Theorem~\ref{T:abe correspondence} involves a $p$-adic replication of Lafforgue's geometric realization of the Langlands correspondence for $\GL_r$ in the cohomology of moduli spaces of shtukas. Consequently, it requires not just the rigid cohomology of smooth varieties with coefficients in overconvergent $F$-isocrystals, but a much more detailed cohomological formalism: constructible sheaves, complexes, nonsmooth varieties, and even algebraic stacks. However, since we take Theorem~\ref{T:abe correspondence} as a black box, we will not encounter any of these subtleties; the reader interested in them may start with the discussion at the end of \cite{kedlaya-isocrystals}.
\end{remark}

\begin{cor} \label{C:finiteness curve}
Suppose that $X$ is proper.
For any positive integer $r$ and any category of coefficient objects on $X$, there are only finitely many
isomorphism classes of irreducible objects of rank $r$ up to twisting by constant objects of rank $1$.
\end{cor}
\begin{proof}
By Lemma~\ref{L:decompose to finite order}, the problem does not change if we consider only objects with determinant of finite order. For such objects, we may apply Theorem~\ref{T:abe correspondence}
to equate the problem to a corresponding finiteness statement for automorphic representations,
for which an argument of Harder applies
\cite[Theorem~1.2.1]{harder}
(compare the discussion in \cite[II.2.1]{yu}).
\end{proof}

\begin{remark}
In light of Corollary~\ref{C:finiteness curve} (and Corollary~\ref{C:finiteness curve2} below), there has been much interest in the question of counting lisse Weil $\overline{\QQ}_\ell$-sheaves on curves with various prescribed properties,
starting with a celebrated note of Drinfeld \cite{drinfeld-number}.
Subsequent work in this direction can be found in \cite{deligne-comptage, deligne-flicker, flicker, flicker2, esnault-kerz, kontsevich, yu}. Most previous progress has been made using automorphic trace formulas; the existence of crystalline companions may provide an alternate approach through the geometry of moduli spaces of vector bundles. For a concrete problem in this direction, we suggest to recover the result of \cite{drinfeld-number} by counting $F$-isocrystals.
\end{remark}

\subsection{Companions of algebraic coefficient objects}

We next address algebraicity and the existence of companions, corresponding to parts (ii), (v), and (vi) of Conjecture~\ref{conj:deligne}. It is convenient to construct companions not only for coefficient objects which are irreducible of finite determinant, but also for algebraic coefficient objects.

\begin{theorem} \label{T:companion dimension 1}
Every coefficient object on $X$ which is irreducible with determinant of finite order is uniformly algebraic and admits companions in all categories of coefficient objects, which are again irreducible with determinant of finite order.
Moreover, the companion in a given category of coefficient objects, for a given embedding of $\overline{\QQ}$
into the full coefficient field, is unique up to isomorphism.
\end{theorem}
\begin{proof}
See \cite[Theorems~4.4.1, 4.4.5]{abe-companion}. Note that \cite[Theorem~4.4.5]{abe-companion} is stated in a conditional form, but is in fact unconditional when $\dim(X) = 1$.
\end{proof}

It is useful to extract the following corollary.
\begin{cor} \label{C:Galois conjugates}
Let $\calE$ be a coefficient object on $X$ which is
$E$-algebraic for some Galois number field $E$.
Then for each $\tau \in \Gal(E/\QQ)$, there exists a coefficient object $\calE_\tau$ on $X$ such that for each $x \in X^\circ$, 
$P((\calE_\tau)_x, T) = \tau(P(\calE_x, T))$ $($where $\tau$ acts coefficientwise on $E[T])$.
\end{cor}
\begin{proof}
Recall that if $E'$ is a Galois number field containing $E$, then $\Gal(E'/\QQ)$ surjects onto $\Gal(E/\QQ)$; this means that there is no harm in enlarging $E$ during the course of the proof.
By Lemma~\ref{L:decompose to finite order}, 
there exists a constant coefficient object $\calL$ of rank 1 such that $\det(\calE \otimes \calL)$ is of finite order. Since $\calE$ is $E$-algebraic for some $E$, $\calL$ is $E$-algebraic for some (possibly larger) $E$; we may thus twist by $\calL$ to reduce to the case where $\calE$ is irreducible and $\det(\calE)$ is of finite order. At this point, the claim is immediate from Theorem~\ref{T:abe correspondence},
or more precisely from the fact that there is such a correspondence for each embedding of $\overline{\QQ}$ into the full coefficient field (see \cite[(2.2)]{deligne-finite} for the corresponding discussion in the \'etale case).
\end{proof}

\begin{cor} \label{C:uniformly algebraic}
For $\calE$ a coefficient object on $X$, the following conditions are equivalent.
\begin{enumerate}
\item[(i)]
$\calE$ is algebraic.
\item[(ii)]
$\calE$ is uniformly algebraic.
\item[(iii)]
Each irreducible constituent of $\calE$ is the twist of an object with determinant of finite order by an algebraic object of rank $1$ pulled back from $k$.
\item[(iv)]
There exists a positive integer $n$ such that the pullback of $\calE$ to $X_n$ satisfies (iii) $($with $k$ replaced by $k_n)$.
\end{enumerate}
\end{cor}
\begin{proof}
It is obvious that (ii) implies (i) and that (iii) implies (iv).
By Theorem~\ref{T:companion dimension 1}, (iv) implies (ii).
By Lemma~\ref{L:decompose to finite order}, (i) implies (iii).
\end{proof}

\begin{cor} \label{C:algebraic companion}
Conjecture~\ref{conj:companion} holds when $\dim(X) = 1$ $($for any $\ell, \ell')$.
\end{cor}
\begin{proof}
We may reduce immediately to the case of an irreducible object;
by Lemma~\ref{L:decompose to finite order}, it must have the form
$\calE \otimes \calL$
where $\calE$ is a coefficient object on $X$ with determinant of finite order and $\calL$
is a constant algebraic coefficient object of rank 1. Note that $\calE$ must itself be irreducible;
hence by Theorem~\ref{T:companion dimension 1}, $\calE$ admits a companion. Meanwhile, $\calL$ trivially admits a companion, as thus does $\calE \otimes \calL$.
\end{proof}

The following statement resembles \cite[Lemma~2.7]{drinfeld-deligne} in the \'etale case; we recast the argument so that it works uniformly, but without taking care to achieve the same bound on the field extension as in \textit{op. cit.} (Note that we do not claim that $L_1 = L_0$.)

\begin{cor} \label{C:bound coefficient extension}
Fix a category of coefficient objects and an embedding of $\overline{\QQ}$ into the full coefficient field.
Let $E$ be a number field and let $L_0$ be the completion of $E$ in the full coefficient field.
Let $\calE$ be an $E$-algebraic coefficient object on $X$ (in the specified category) of rank $r$.
Then there exists a finite extension $L_1$ of $L_0$, depending only on $L_0$ and $r$ (but not on $X$ or $\calE$ or $E$),
for which $\calE$ can be realized as an object of $\Weil(X) \otimes L_1$ or $\FIsoc^\dagger(X) \otimes L_1$.
\end{cor}
\begin{proof}
There is no harm in enlarging $L_0$ to ensure that $\mu_r \subset L_0$.
We may further assume that $\calE$ is irreducible (of rank $r$) with determinant of finite order,
as then for general $\calE$ we may decompose into irreducibles and apply Lemma~\ref{L:decompose to finite order} to obtain the original statement (but with a larger bound on $L_1$).
Let $n$ be the order of $\det(\calE)$; since $L_0$ contains only finitely many roots of unity, $n$ is bounded in terms of $L_0$ alone.

For $L$ a finite extension of $L_0$, let $\calC_L$ denote the category $\Weil(X) \otimes L$ or $\FIsoc^\dagger(X) \otimes L$ within the full category of coefficient objects under consideration.
By definition, we can realize $\calE$ in $\calC_L$  for some finite extension $L$ of $L_0$, which we may assume is Galois over $L_0$.
Since $\calE$ remains irreducible upon enlarging $L$, by Schur's lemma the endomorphism algebra of $\calE$ in $\calC_L$ is equal to $L$.

Put $G := \Gal(L/L_0)$; then $G$ acts on the category $\calC_L$,
so we may form the pullback $\tau^* \calE$ for any $\tau \in G$. By the uniqueness aspect of
Theorem~\ref{T:companion dimension 1},
we may choose an isomorphism $\iota_\tau: \tau^* \calE \cong \calE$ for each $\tau$.
For $\tau_1, \tau_2 \in G$, 
the isomorphism $\iota_{\tau_1 \tau_2}: (\tau_1 \tau_2)^* \calE \to \calE$ differs 
from the composition $ \iota_{\tau_1} \circ \tau_1^*(\iota_{\tau_2})$ by multiplication by a nonzero scalar
$c(\tau_1,\tau_2) \in L$. The $c_{\tau_1, \tau_2}$ form a 2-cocycle and thus represent a class in $H^2(G, L^\times)$;
since $\det(\calE)$ has order $n$, this class in fact belongs to $H^2(G, \mu_{rn})$. By local class field theory, this class (which represents an $rn$-torsion element of the Brauer group of $L_0$) can be trivialized by passing from $L_0$ to any finite extension $L_1$ of $L_0$ of degree $rn$ (\textit{e.g.} the unramified one); this proves the claim.
\end{proof}

\subsection{Proof of  Theorem~\ref{T:one-dim case}}
\label{subsec:one-dim}

We now complete the proof of Theorem~\ref{T:one-dim case}. For this, in light of Theorem~\ref{T:companion dimension 1}
we could simply assume $\ell \neq p$ and appeal to \cite[Th\'eor\`eme~VII.6]{lafforgue}; however, we prefer to spell out a bit more explicitly
how the various aspects of Conjecture~\ref{conj:deligne} are addressed. Let $r$ be the rank of $\calE$.

\begin{enumerate}
\item[(i)]
By Theorem~\ref{T:companion dimension 1}, we may assume that $\ell \neq p$ and that $\calE$ is uniformly algebraic.
Using Corollary~\ref{C:Galois conjugates}, we may take the direct sum of the Galois conjugates of $\calE$ to obtain a lisse Weil $\overline{\QQ}_\ell$-sheaf $\calF$ which is $\QQ$-algebraic. In particular, 
for any algebraic embedding $\iota$ of $\overline{\QQ}_\ell$ into $\CC$, 
$\calF$ is $\iota$-real, and so we may apply 
\cite[Th\'eor\`eme~1.5.1]{deligne-weil2} to deduce that its constituents are all $\iota$-pure.
In particular $\calE$ is $\iota$-pure of some weight $w$; its determinant is then pure of weight $r w$,
but this forces $w=0$ because the determinant is of finite order.

\item[(ii)]
This is included in Theorem~\ref{T:companion dimension 1}.

\item[(iii)]
By Theorem~\ref{T:companion dimension 1}, $\calE$ admits an $\ell'$-adic companion where $\ell'$ is the rational prime below $\lambda$. Since this companion is irreducible with finite determinant, it must be a lisse 
 \'etale $\overline{\QQ}_{\ell'}$-sheaf  (see Definition~\ref{D:lisse etale}), and this proves the claim.
 
\item[(iv)]
By Theorem~\ref{T:companion dimension 1}, we may assume that $\ell \neq p$.
Applying \cite[Th\'eor\`eme~VII.2]{lafforgue} yields an upper bound of the desired form, except that the desired factor of $\frac{r}{2}$ is replaced by the larger factor $\frac{(r-1)^2}{r}$. 

On the other hand,
combining Theorem~\ref{T:abe correspondence} with \cite[Corollaire~2.2]{lafforgue-hecke} (which makes the calculation on the automorphic side) yields a bound with the factor $\frac{r-1}{2}$; moreover, it is shown that for $i=0,\dots,r$, the sum of the $i$ largest valuations is at most
$\frac{i(r-i)}{2}$ times the valuation of $\# \kappa(x)$. 
Another way to get the same bound is to apply Theorem~\ref{T:companion dimension 1} to reduce to the case $\ell = p$
and then invoke Proposition~\ref{P:uniform bound on slopes}; see \cite[Remark~1.3.6]{drinfeld-kedlaya}.

\item[(v, vi)]
These are both included in Theorem~\ref{T:companion dimension 1}.

\end{enumerate}

\subsection{Compatibility of ramification}

We now spell out in more detail what the compatibility assertion of Theorem~\ref{T:abe correspondence} says about the local behavior of companions. For this, we need to introduce local monodromy representations, particularly in the crystalline case.

\begin{defn}
For $x \in \overline{X}$, let $I_x$ denote the inertia group of $\overline{X}$ at $x$.
For $\calE$ a coefficient object, denote by $\rho_x(\calE)$ the \emph{semisimplified local monodromy representation} of $I_x$
associated to $\calE$. In the \'etale case, this is simply the semisimplification of the restriction of the associated Weil group representation; in the crystalline case, see \cite[Remark~4.12]{kedlaya-isocrystals}.
We need the following properties of the construction.
\begin{enumerate}[label=$\bullet$]
\item
$\calE$ is tame at $x$ if and only if $\rho_x(\calE)$ is at most 
tamely ramified. In particular, if we write $\Swan_x(\calE)$ for the Swan conductor of $\rho_x(\calE)$, then $\calE$ is tame at $x$ if and only if $\Swan_x(\calE) = 0$.
\item
$\calE$ is docile at $x$ if and only if $\rho_x(\calE)$ is unramified. In particular, this is true if $x \in X$.
\item
The construction of $\rho_x$ is functorial in the following sense: if $f: \overline{X}' \to \overline{X}$ is a finite flat morphism and $f(y) = x$, then $\rho_y(f^* \calE)$ is the restriction of $\rho_x(\calE)$ along $I_y\to I_x$.
\end{enumerate}
\end{defn}

For curves, the Euler characteristic is related to local monodromy via the Grothendieck--Ogg--Shafarevich formula.
\begin{prop} \label{P:GOS formula}
For any coefficient object $\calE$ on $X$, we have 
\[
\chi(\calE) = \chi(X) \rank(\calE) - \sum_{x \in Z} [\kappa(x)\colon\! k] \Swan_x(\calE) .
\]
\end{prop}
\begin{proof}
See \cite{grothendieck-gos} in the \'etale case and 
\cite[Theorem~4.3.1]{kedlaya-weil2} in the crystalline case.
\end{proof}
\begin{cor} \label{C:tame compatibility}
Let $\calE$ and $\calF$ be coefficient objects on $X$ $($possibly in different categories$)$ which are companions.
Then
$\calE$ is tame $($resp.\ docile$)$ if and only if $\calF$ is.
$($By Lemma~\ref{L:extend tame}, this remains true even without  the same conclusion holds without Hypothesis~\ref{H:curve}, i.e., allowing $X$ to be smooth but not necessarily one-dimensional.$)$
\end{cor}
\begin{proof}
From Definition~\ref{D:L-function} and Proposition~\ref{P:GOS formula}, we can read off the (weighted) sum of Swan conductors of $\calE$ from $L(\calE^\dual, T)$; in particular, the $L$-function determines whether or not this sum is zero, and hence whether or not $\calE$ is tame. Since the duals of companions have matching $L$-functions, this proves the claim in the tame case.

In the docile case, we may assume at once that both $\calE$ and $\calF$ are tame;
we may also assume that they are both irreducible with determinant of finite order. 
For each $x \in \overline{X}$, by Theorem~\ref{T:abe correspondence} the local Euler factors of $\calE$ and $\calF$ at $x$ coincide. The degree of the local Euler factor computes the dimension of the kernel of either the full local monodromy representation in the \'etale case, or the monodromy operator in the crystalline case (see Remark~\ref{R:local Euler factors}). Let $d_1$ denote this dimension, let $d_2$ denote the corresponding dimension for $\calE^\dual$,
and let $r$ denote the common rank of $\calE$ and $\calF$. Then the degree of the local Euler factor of $\calE^\dual \otimes \calE$ at $x$ is at least $d_1r + d_2r - d_1 d_2$, with equality if and only if $\calE$ is docile.
(This quantity computes the contribution of the generalized eigenspace for the eigenvalue 1; any other eigenvalue that occurs makes a positive contribution.) By applying the same logic with $\calE$ replaced by $\calF$, we deduce the claim.
\end{proof}

We next formally promote Corollary~\ref{C:tame compatibility} to a statement about ramification at individual points.

\begin{lemma} \label{L:modify wild ramification}
Let $f: \overline{X}' \to \overline{X}$ be a finite flat morphism which is \'etale over $X$,
and choose $x \in Z$. Then there exists a finite flat morphism $g: \overline{X}'' \to \overline{X}$ with the following properties.
\begin{enumerate}
\item[(i)]
$g$ is \'etale over $x$.
\item[(ii)]
For each $x' \in Z \setminus \{x\}$, write $Z_{x'}$ for the formal completion of $\overline{X}$ along $x'$; then the base extension $\overline{X}'' \times_{\overline{X}} Z_{x'} \to Z_{x'}$ of $g$
factors through the base extension $\overline{X}' \times_{\overline{X}} Z_{x'} \to Z_{x'}$ of $f$.
\end{enumerate}
$($Note that we do not require $g$ to be \'etale over $X.)$
\end{lemma}
\begin{proof}
It suffices to achieve (i) and (ii) for a single $x'$, as we may then take a fiber product of the resulting covers to achieve the desired result. For this, first choose a finite flat morphism from $\overline{X}$ to $\PP^1_k$ carrying $x$ and $x'$ to $1$ and $0$, respectively. We may then achieve the result using the Katz--Gabber canonical extension theorem \cite{katz-local}: any extension of $k((t))$ can be uniquely realized as the base extension of a finite flat cover of $\PP^1_k$ which is tamely ramified at $t = \infty$ and \'etale away from $t=0,\infty$. 
\end{proof}

\begin{prop} \label{P:compare tame docile}
Let $\calE$ and $\calF$ be coefficient objects on $X$ $($possibly in different categories$\,)$ which are companions.
Then for $x \in Z$, the following statements hold.
\begin{enumerate}
\item[(a)]
We have $\Swan_x(\calE) = \Swan_x(\calF)$. In particular, $\calE$ is tame at $x$ if and only if $\calF$ is.
\item[(b)]
$\calE$ is docile at $x$ if and only if $\calF$ is.
\end{enumerate}
\end{prop}
\begin{proof}
Apply Remark~\ref{R:tamely} to construct a finite flat morphism
$f: \overline{X}' \to \overline{X}$ which is \'etale over $X$ 
such that both $f^* \calE$ and $f^* \calF$ are docile.
(More precisely, by Corollary~\ref{C:algebraic companion} and
Corollary~\ref{C:tame compatibility} we may do this assuming that both $\calE$ and $\calF$ are \'etale coefficient objects, and then Remark~\ref{R:tamely} applies.)
Define another morphism $g: \overline{X}' \to \overline{X}$ as in Lemma~\ref{L:modify wild ramification}.
Then $g^* \calE$ and $g^* \calF$ are both tame at every point of $\overline{X}'$ not lying over $x$;
meanwhile, for $y \in g^{-1}(x)$, we have $\Swan_y(g^*\calE) = \Swan_x(\calE)$ and $\Swan_y(g^*\calF) = \Swan_x(\calF)$. We may thus argue using Proposition~\ref{P:GOS formula} (as in the proof of Corollary~\ref{C:tame compatibility}) to deduce (a). By a similar argument, we may reduce (b) to Corollary~\ref{C:tame compatibility}.
\end{proof}

\begin{cor} \label{C:finiteness curve2}
For any positive integer $r$ and any category of coefficient objects on $X$, there are only finitely many
isomorphism classes of irreducible tame objects of rank $r$ up to twisting by constant objects of rank $1$.
\end{cor}
\begin{proof}
By Lemma~\ref{L:decompose to finite order}, the problem does not change if we consider only objects with determinant of finite order. 
By Theorem~\ref{T:companion dimension 1} and Proposition~\ref{P:compare tame docile}, we need only consider the \'etale case.
In this case, an irreducible coefficient object of rank $r$ corresponds via Theorem~\ref{T:companion dimension 1}
to an automorphic representation.
If the original coefficient object is tame, then by local-global compatibility, the resulting automorphic representation has Iwahori level,
which in particular depends only on $r$; hence \cite[Theorem~1.2.1]{harder} again applies.
\end{proof}

\section{Higher-dimensional varieties}

In this section, we  prove Theorem~\ref{T:deligne conj2}, thus resolving Conjecture~\ref{conj:deligne}(i)--(v)
for general $X$.
As noted in the introduction, this includes some duplication of results of Abe--Esnault \cite{abe-esnault}; we comment at the end of the section on differences between the two approaches (see \S\ref{sec:abe-esnault}).

\begin{convention}
We will frequently do the following: pick some closed points $x_1,\dots,x_m \in X^\circ$, choose a positive integer $n$, and choose a ``curve in $X_n$ containing $x_1,\dots,x_m$.'' By this, we mean a curve (over $k_n$) in $X_n$  containing some points lying over $x_1,\dots,x_m$; we remind the reader that we do not insist that curves be proper,
so the existence of such a curve is elementary (that is, it does not depend on anything as difficult as Poonen's Bertini theorem \cite{poonen}).
\end{convention}

\subsection{Weights}
\label{subsec:weight filtration}

We first address part (i) of Conjecture~\ref{conj:deligne} by consolidating Deligne's theory of weights from \cite{deligne-weil2}
using the Langlands correspondence in the form of Theorem~\ref{T:one-dim case}.
In the process, we will make a first use of the geometric setup that will recur in the discussion of Lefschetz slicing
(\S\ref{sec:slicing}).

\begin{hypothesis}
Throughout \S\ref{subsec:weight filtration},
fix a coefficient object $\calE$ on $X$,
and let $\iota$ be an embedding of the full coefficient field of $\calE$ into $\CC$.
(This choice is for expediency only; see Remark~\ref{R:algebraic embedding}.)
Also, let $q$ denote the cardinality of $k$.
\end{hypothesis}

\begin{defn}
For $x \in X^\circ$, the \emph{multiset of $\iota$-weights of $\calE$ at $x$}
is the multiset consisting of $-2 \log_{\#\kappa(x)} \left| \iota(\lambda) \right|$ as $\lambda$ varies over the roots of $P(\calE_x, T)$ (counted with multiplicity). Recall that these roots are reciprocals of eigenvalues of Frobenius.

We say that $\calE$ is \emph{$\iota$-pure of weight $w$} if for all $x \in X^\circ$, the multiset of $\iota$-weights of $\calE$ at $x$ consists of the single element $w$. If the determinant of $\calE$ is of finite order, we must then have $w=0$.
\end{defn}

We will make crucial use of the following form of ``Weil II''.
\begin{lemma} \label{L:Weil II}
Suppose that $\calE$ is $\iota$-pure of weight $w$. Then the eigenvalues of the Frobenius $\varphi$ on $H^1(X, \calE)$ all have $\iota$-absolute value at least $q^{(w+1)/2}$.
\end{lemma}
\begin{proof}
See \cite[Th\'eor\`eme~3.3.1]{deligne-weil2} in the \'etale case and 
\cite[Theorem~10.3]{kedlaya-isocrystals} in the crystalline case.
\end{proof}

As in \cite[Th\'eor\`eme~3.4.1(3)]{deligne-weil2}, we have the following corollary.
(See also \cite[Corollary~3.5.2]{daddezio}.)
\begin{cor} \label{C:geometrically semisimple}
Suppose that $\calE$ is $\iota$-pure.
\begin{enumerate}
\item[(a)]
In the category of geometric coefficient objects, $\calE$ is semisimple.
\item[(b)]
In the category of coefficient objects, $\calE$ admits an \emph{isotypical decomposition}: a direct sum decomposition in which each summand is a successive extension of a single irreducible object.
\end{enumerate}
\end{cor}
\begin{proof}
We prove both assertions at once by induction on $\rank(\calE)$.
Suppose first that $\calE$ is irreducible. Let $\calF$ be the sum of all irreducible subobjects of $\calE$ in the category of geometric coefficient objects; then $\calF$ is stable under Frobenius, and hence a subobject of $\calE$ in the category of coefficient objects. It follows that $\calE = \calF$, proving (a); meanwhile, (b) is vacuous.

Suppose next that $\calE$ is reducible. Then there exists
an exact sequence
\begin{equation} \label{eq:sequence coefficient objects}
0 \to \calE_1 \to \calE \to \calE_2 \to 0
\end{equation}
of coefficient objects in which $\calE_1$ is irreducible and $\calE_2$ is nonzero.
If this sequence splits, then we may apply the induction hypothesis to $\calE_2$ to deduce both (a) and (b). Otherwise, from the Hochschild--Serre spectral sequence (compare
\cite[Lemme~3.4.2]{deligne-weil2} in the \'etale case), we obtain an exact sequence
\begin{equation} \label{eq:sequence coefficient objects2}
H^0(X, \calE_2^\dual \otimes \calE_1)_{\varphi} \to \Ext(\calE_2, \calE_1)
\to H^1(X, \calE_2^\dual \otimes \calE_1)^{\varphi}
\end{equation}
where the $\Ext$ group is defined in the category of coefficient objects.
Since $\calE_2^\dual \otimes \calE_1$ is $\iota$-pure of weight $0$, the term on the right vanishes by Lemma~\ref{L:Weil II}.
It follows that the extension \eqref{eq:sequence coefficient objects} splits in the category of geometric coefficient objects,
so we may deduce (a) from the induction hypothesis.
Meanwhile, the nonvanishing of the middle term of \eqref{eq:sequence coefficient objects2} implies that the finite-dimensional vector space $H^0(X, \calE_2^\dual \otimes \calE_1)$ has nonzero $\varphi$-coinvariants and hence nonzero $\varphi$-invariants; that is, there is a nonzero homomorphism $\calE_2 \to \calE_1$
in the category of coefficient objects (not just geometric coefficient objects).
By the induction hypothesis, $\calE_2$ admits an isotypical decomposition, which must admit a nonzero summand corresponding to $\calE_1$. Let $\calE'_1$ be the preimage of this summand in $\calE$; then the same argument applied to the exact sequence
\[
0 \to \calE'_1 \to \calE \to \calE'_2 \to 0
\]
shows that this sequence must split. Since $\calE'_1$ is isotypical, we may apply the induction hypothesis to $\calE'_2$ to deduce (b).
\end{proof}

The following lemma will occur in passing in the construction of the weight filtration, but will be used
much more heavily in the study of uniqueness of companions (\S\ref{subsec:Chebotarev density}).
\begin{lemma} \label{L:pole order}
Suppose that $X$ is irreducible of pure dimension $d$ and that  $\calE$ is
$\iota$-pure of weight $0$. Then
the pole order of $L(\calE, T)$ at $T = q^{-d}$ equals the multiplicity of $1$ as an eigenvalue of Frobenius on $H^0(X, \calE^\dual)$ $($viewed as a vector space over the full coefficient field$)$.
\end{lemma}
\begin{proof}
In the Lefschetz trace formula \eqref{eq:lefschetz}, the factor $i=0$ contributes the predicted value to the pole order. Meanwhile, by Lemma~\ref{L:Weil II}, the eigenvalues of $F$ on $H^i(X, \calE^\dual)$ for $i>0$ all have absolute value at least $q^{1/2}$,
so the corresponding factor of \eqref{eq:lefschetz} only has zeroes or poles in the region
$|T| \geq q^{-d+1/2}$. This proves the claim.
\end{proof}

By probing $X$ with curves, we gain some initial data about weights.

\begin{lemma} \label{L:pointwise weights}
Suppose that $X$ is irreducible.
\begin{enumerate}
\item[(a)]
The multiset of $\iota$-weights of $\calE$ at $x \in X^\circ$ is independent of $x$.
We may thus refer to it as the \emph{multiset of $\iota$-weights of $\calE$}.
\item[(b)]
There exist a constant twist $\calE'$ of $\calE$ and a positive integer $N$ such that 
the multiset of $\iota$-weights of $\calE'$ has least element $0$
and for each $x \in X^\circ$, the product of the roots of $P(\calE'_x, T)$ of $\iota$-weight $0$
equals a root of unity of order dividing $N$.
\end{enumerate}
\end{lemma}
\begin{proof}
To check (a), it suffices to compare two points $x,y \in X^\circ$; this may be done by restricting to a curve in $X$ through 
$x,y$. That is, we may assume that $X$ is a curve; in this case, we may further assume that $\calE$ is irreducible and
(by making a constant twist dictated by 
Lemma~\ref{L:decompose to finite order}) its determinant is of finite order. We may then apply Theorem~\ref{T:one-dim case}
to deduce that $\calE$ is $\iota$-pure of weight $0$.

To check (b), we may first twist to ensure that for some $x \in X^\circ$, the product in question equals $1$.
Given any other point $y \in X^\circ$, choose a curve $C$ in $X$ containing $x$ and $y$; by Theorem~\ref{T:one-dim case} again, 
the eigenvalues in question arise from some irreducible constituents of $\calE|_C$. By Lemma~\ref{L:decompose to finite order},
we see that the product of the eigenvalues at $y$ is a root of unity; since these eigenvalues belong to an extension of $\QQ_\ell$ or $\QQ_p$ of degree bounded independently of $x$, the order of the root of unity is also bounded.
\end{proof}

\begin{cor} \label{C:pure after restriction}
Suppose that for some $w \in \RR$,
there exists an open dense subset $U$ of $X$ such that $\calE|_U$ is $\iota$-pure of weight $w$.
Then $\calE$ is also $\iota$-pure of weight $w$.
\end{cor}
\begin{proof}
This is immediate from Lemma~\ref{L:pointwise weights}(a).
\end{proof}

We now introduce the geometric setup for Lefschetz slicing.
We recall here a commonly used construction found in \cite[Proposition~3.3]{artin-sga4}.
(One can make a similar construction without enlarging $k$ using Poonen's Bertini theorem; see \S\ref{sec:abe-esnault}.)
\begin{lemma} \label{L:elementary fibration}
Pick $x \in X(k)$ and let
$\tilde{X}$ be the blowup of $X$ at $($the reduced closed subscheme supported at$)$ $x$. 
Then for some finite extension $k'$ of $k$,
there exists an open neighborhood $U$ of $x$ in $X_{k'}$ whose inverse image $\tilde{U}$ in $\tilde{X}_{k'}$ can be expressed
as a fibration $f: \tilde{U} \to S$ in curves over some smooth proper $k'$-scheme $S$.
\end{lemma}
\begin{proof}
Since we are free to replace $X$ with an open subscheme containing $x$, we may ensure that $X$ is affine of dimension $n$.
Let $\overline{X}$ be the Zariski closure of $X$ in some projective embedding $X \hookrightarrow \PP^N_k$.
For some choice of $k'$, we can choose sections $s_0,\dots,s_{n-1}$ of $\calO(1)$ on $\PP^N_{k'}$ whose zero loci $H_0,\dots,H_{n-1}$ on $\PP^N_{k'}$ have the following properties
(because each property holds generically).
\begin{enumerate}[label=$\bullet$]
\item
The intersection $L = H_0 \cap \cdots \cap H_{n-1}$ has dimension $N-n$ and contains $x$.
\item
The intersection $L \cap X_{k'}$ is transversal at $x$.
\item
The intersection $L \cap (\overline{X}_{k'} \setminus X_{k'})$ is empty.
\end{enumerate}
Put $Z := (L \cap X_{k'}) \setminus \{x\}$;
then the sections $s_0,\dots,s_{n-1}$ define a map
$f: \tilde{X}_{k'} \setminus Z \to S := \PP^{n-1}_{k'}$.
The exceptional divisor of $\tilde{X}_{k'} \to X_{k'}$ 
is the image of a section $g: S \to \tilde{X}_{k'}$ of $f$.
Each fiber of $f$ is isomorphic to an intersection of $X_{k'} \setminus Z$ with $n-1$ hyperplanes which is transversal at $x$; consequently, there is a unique irreducible component of this fiber containing $x$, and this component is smooth of dimension 1 at $x$.
We now take $\tilde{U}$ to be the set obtained from $\tilde{X}_{k'} \setminus Z$ by removing:
\begin{enumerate}[label=$\bullet$]
\item
the union of all irreducible components of fibers of $f$ which do not meet $g(S)$; and
\item
the set of all singular points of irreducible components of fibers of $f$ which do meet $g(S)$.
\end{enumerate}
This yields an open subscheme $\tilde{U}$ of $\tilde{X}_{k'} \setminus Z$ not meeting $g(S)$;
we take $U$ to be the image of $\tilde{U}$ in $X_{k'}$.
\end{proof}

\begin{remark} \label{R:elementary fibration}
We iterate a key point from the proof of Lemma~\ref{L:elementary fibration}: the exceptional divisor of the blowup $\tilde{U} \to U$ is the image of a section $g: S \to \tilde{U}$ of $f$
on which the pullback of $\calE$ is constant.
We will combine this geometric construction with the following observation: for any coefficient object $\calF$ in the same category as $\calE$, there exists an open dense subscheme $V$ of $S$ (depending on $\calF$)
such that $f_* \calF$ exists as a coefficient object on $V$ and its formation commutes with arbitrary base change on $V$.
(See \cite[Theorem~7.3.3]{kedlaya-finiteness} for the crystalline case.)
\end{remark}

We finally obtain the main theorem on weights.
\begin{theorem} \label{T:weight filtration}
The following statements hold.
\begin{enumerate}
\item[(a)]
If $\calE$ is irreducible, then it is $\iota$-pure of some weight.
\item[(b)]
There exists a unique filtration
\[
0 = \calE_0 \subset \cdots \subset \calE_l = \calE
\]
for which there exists an increasing sequence $w_1 < \cdots < w_l$ of real numbers
for which $\calE_i/\calE_{i-1}$ is $\iota$-pure of weight $w_i$. We call this filtration the \emph{weight filtration} of $\calE$.
\end{enumerate}
\end{theorem}
\begin{proof}
For any given $X$, part (a) of Theorem~\ref{T:weight filtration} implies part (b) using Lemma~\ref{L:Weil II};
see \cite[Corollary~10.5]{kedlaya-isocrystals} or  \cite[Corollary~3.5.2]{daddezio}.
For $\dim(X) = 1$, we may apply Theorem~\ref{T:one-dim case} to obtain (a). 

It thus remains to prove (a) for general $X$
using the truth of (b) for curves. For this, we may assume that $X$ is irreducible; moreover,
by Corollary~\ref{C:pure after restriction}, we may check the claim after replacing $X$ with an open dense subscheme.
We may also make a constant twist on $\calE$; hence by Lemma~\ref{L:pointwise weights},
we may assume that the multiset of $\iota$-weights of $\calE$ has least element $0$ with some multiplicity $m$, and that for each $x\in X^\circ$ the product of
the corresponding roots of $P(\calE_x,T)$ is a root of unity of order dividing some fixed positive integer $N$.

After shrinking $X$ and enlarging $k$, we may apply Lemma~\ref{L:elementary fibration} and Remark~\ref{R:elementary fibration} to $\calF := \wedge^{mN} (\calE^{\oplus N})$. 
On each fiber of $f$, we may apply the induction hypothesis to produce a weight filtration of $\calF$; the first step
of this filtration is an object of rank $1$ with all Frobenius eigenvalues equal to 1, so by Lemma~\ref{L:pole order} 
it is trivial. 
On some open subscheme $V$ of $S$, we may form the pushforward $f_* \calF$ in a manner compatible with arbitrary base change (Remark~\ref{R:elementary fibration});
this pushforward must therefore be trivial of rank 1, and we may pull back to obtain
a trivial rank-$1$ subobject of $\calF$ (initially on $f^{-1}(V)$, but by
Lemma~\ref{L:restriction functor} this subobject extends over $X$). 
By multilinear algebra, we recover a subobject of $\calE$ of rank $m$;
this is a contradiction unless $m = \rank(\calE)$, meaning that $\calE$ is $\iota$-pure of weight $0$.
\end{proof}

\begin{remark}
An alternate approach to proving Theorem~\ref{T:weight filtration} would be to use the fact that
part (a) can be checked after an alteration; in light of Proposition~\ref{P:alter to tame}, this means that it suffices to check (a) under the additional assumption that $\calE$ is tame. In this context, one can obtain very strong versions of Lefschetz
slicing; see \S\ref{sec:abe-esnault}.
\end{remark}

\begin{remark}\label{R:algebraic embedding}
We again remind the reader of \cite[Remarque~1.2.11]{deligne-weil2}:
the use of an algebraic embedding $\iota$ in Theorem~\ref{T:weight filtration} is merely for expository convenience.
\emph{A posteriori}, we will only need to embed various number fields into $\CC$.

On a similar note, the weight of an irreducible coefficient object may depend on the choice of $\iota$,
but in light of Lemma~\ref{L:decompose to finite order} this reduces to a rank-1 issue.
In any case, note that changing $\iota$ may change the order of the successive quotients in the weight filtration,
in which case some splitting must occur.
\end{remark}

\subsection{Lefschetz slicing and valuations of eigenvalues}
\label{sec:slicing}

As discussed in the introduction, a key tool in reducing the study of general coefficient objects to the case of curves is the preservation of irreducibility of a coefficient object under restriction to a suitable curve,
but the usual proof of this in the \'etale case does not apply in the crystalline case.
We give a replacement argument here using weights; we make no effort to strengthen the final result because this can easily be done
\emph{a posteriori} (see \S\ref{sec:abe-esnault}).
We then deduce parts (iii) and (iv) of Conjecture~\ref{conj:deligne}, as well as a weak version of (ii).

\begin{lemma} \label{L:slicing}
Let $\calE$ be a geometrically irreducible coefficient object on $X$
and choose $x \in X^\circ$.
Then for some positive integer $n$, there exists a curve $C$ in $X_n$ containing $x$ such that $\calE|_C$ is irreducible.
\end{lemma}
\begin{proof}
Instead of keeping track of the extension degree $n$, we adopt the convention that we may enlarge $k$ freely during the argument, replacing $x$ with some closed point lying over it.
To begin with, we may assume that $x \in X(k)$.

Apply Lemma~\ref{L:elementary fibration} and Remark~\ref{R:elementary fibration} with $\calF$ equal to the trace-zero component of $\calE^\dual \otimes \calE$. 
By Remark~\ref{R:absolutely irreducible}, $\overline{G}(\calE)$ must act irreducibly on its standard representation, so 
$H^0(X, \calE^\dual \otimes \calE)$ is the span of the identity map and $H^0(X, \calF) = 0$.

Since $g^* \calE$ is a trivial coefficient object, so is $g^* \calF$. 
By adjunction, $f^* f_* \calF$ corresponds to a subobject of $\calF$, so $g^* f^* f_* \calF$ corresponds to a nonzero subobject of $g^* \calF$;
hence $g^* f^* f_* \calF$ is trivial.
Since $g$ is a section of $f$, there is a natural identification $g^* f^* f_* \calF \cong f_* \calF$; hence $f_* \calF$ is trivial, as then is $f^* f_* \calF$.
In particular, the dimension of $H^0(X, f^* f_* \calF)$ equals the rank of $f_* \calF$; since $H^0(X, f^* f_* \calF)$ is a subspace of $H^0(X, \calF) = 0$, we conclude that $f_* \calF = 0$.

Choose a point $y \in S(k)$ and take $C := f^{-1}(y)$. The vanishing of $f_* \calF$ implies that $H^0(C, \calF|_C) = 0$.
Pick an embedding $\iota$ of the coefficient field of $\calE$ into $\CC$; by Theorem~\ref{T:weight filtration}, $\calE$ is 
$\iota$-pure of some weight, and likewise after restricting to $C$. By Corollary~\ref{C:geometrically semisimple}, $\calE|_C$
admits an isotypical decomposition; combined with the vanishing of $H^0(C, \calF|_C)$, this implies first that $\calE|_C$ is isotypical, and second that $\calE|_C$ is irreducible.
\end{proof}

\begin{remark}
Another way to establish Lemma~\ref{L:slicing}, which does not require a prior study of weights,
is to 
show that  the inclusion $\overline{G}(\calE|_C) \to \overline{G}(\calE)$ is an isomorphism
by applying Lemma~\ref{L:elementary fibration} and Remark~\ref{R:elementary fibration} with $\calF$ being 
a sufficiently large object in the Tannakian category generated by $\calE$.
This is in the spirit of Corollary~\ref{C:same monodromy} below, but requires slightly more care because the open set $U$
varies with $\calF$; one must thus make some argument to give an \emph{a priori} choice of $\calF$ that suffices to rule out all unwanted possibilities for $\overline{G}(\calE|_C)$.
(As in Lemma~\ref{L:pointwise weights}, the key point is that an extension of $\QQ_\ell$ or $\QQ_p$ of bounded degree contains a bounded number of roots of unity.)
\end{remark}

\begin{lemma} \label{L:valuations}
Let $\calE$ be a coefficient object on $X$ which is irreducible with determinant of finite order.
Let $\lambda$ be a reciprocal root of $P(\calE_x, T)$ for some $x \in X^\circ$.
\begin{enumerate}
\item[(a)]
$\lambda$ is algebraic over $\QQ$.
\item[(b)]
At any $($finite or infinite$)$ place of $\QQ(\lambda)$ not lying above $p$, $\lambda$ has absolute value $1$.
$($For infinite places, this is included in Theorem~\ref{T:weight filtration}.$)$
\item[(c)]
At any place of $\QQ(\lambda)$ lying above $p$, $\lambda$ has absolute value at least that of $\kappa(x)$ to the power $\frac{1}{2} \rank(\calE)$.
\end{enumerate}
In particular, parts (i), (iii), and (iv) of Conjecture~\ref{conj:deligne} hold.
\end{lemma}
\begin{proof}
We start with some initial reductions.
For $n$ a positive integer, let $f_n: X_n \to X$ be the canonical morphism. For $x_n \in X_n^\circ$ lying over $x \in X^\circ$,
the roots of $P((f_n^* \calE)_{x_n}, T)$ are the $[\kappa(x_n):\kappa(x)]$-th powers of the roots of $P(\calE_x, T)$;
consequently, for any particular $x$, checking the claims at $x$ is equivalent to checking them at $x_n$.

This means that by Remark~\ref{R:absolutely irreducible} and Lemma~\ref{L:decompose to finite order2}, we may reduce to the case where $\calE$ is geometrically irreducible. Then for each $x \in X^\circ$, we may apply Lemma~\ref{L:slicing} to find a curve $C$ in $X_n$ through $x$, for some $n$ \emph{depending on $x$}, for which $\calE|_C$ is again irreducible.
We may then invoke Theorem~\ref{T:one-dim case} to deduce that the desired assertions hold at $x$.
\end{proof}

\subsection{Uniqueness of companions via weights}
\label{subsec:Chebotarev density}

For \'etale coefficient objects, the uniqueness of a companion in a given category is an easy consequence of the Chebotar\"ev density theorem applied to residual representations. In order to obtain similar assertions that also cover crystalline coefficient objects, one must replace the use of Chebotar\"ev density with a weight-theoretic argument. The argument we have in mind is due to Tsuzuki (as presented in \cite[Proposition~A.4.1]{abe-companion}, \cite[Theorem~10.8]{kedlaya-isocrystals}), but it must be noted that similar arguments occur in \'etale cohomology in the context of independence-of-$\ell$ statements, as in Larsen--Pink \cite[Proposition~2.1]{larsen-pink} (see also \cite[Proposition~8.20]{pal}).

\begin{theorem}[Tsuzuki] \label{T:chebotarev}
Let $\calE_1, \calE_2$ be two algebraic coefficient objects on $X$ which are companions.
\begin{enumerate}
\item[(a)]
If $\calE_1$ is irreducible, then so is $\calE_2$.
\item[(b)]
If $\calE_1, \calE_2$ are in the same category, then $\calE_1$ and $\calE_2$ have the same semisimplification.
\end{enumerate}
\end{theorem}
\begin{proof}
In both cases, we may assume that $X$ is irreducible of pure dimension $d$, and
(by Theorem~\ref{T:weight filtration}) that $\calE_1, \calE_2$ are both pure of weight 0 and semisimple.

In case (a), note that by Schur's lemma, $\calE_i$ is irreducible if and only if $H^0(X, \calE_i^\dual \otimes \calE_i)^\varphi$ is one-dimensional. We may thus apply Lemma~\ref{L:pole order} to $\calE_1^\dual \otimes \calE_1, \calE_2^\dual \otimes \calE_2$ to conclude.

In case (b), it suffices to check that any irreducible subobject $\calF$ of $\calE_1$ (which must also be pure of weight 0) also occurs as a summand of $\calE_2$.
To this end, note that $\calF$ occurs as a summand of $\calE_i$ if and only if $H^0(X, \calF^\dual \otimes \calE_i)^{\varphi} \neq 0$;
we may thus apply Lemma~\ref{L:pole order} to $\calE_1^\dual \otimes \calF, \calE_2^\dual \otimes \calF$ to conclude.
\end{proof}

\begin{cor} \label{C:extend companions}
Let $\calE$ and $\calF$ be algebraic coefficient objects on $X$ (not necessarily in the same category).
If there exists an open dense subscheme $U$ of $X$ such that $\calE|_U$ and $\calF|_U$ are companions,
then $\calE$ and $\calF$ are companions.
\end{cor}
\begin{proof}
It suffices to check the equality $P(\calE_x, T) = P(\calF_x, T)$ at an arbitrary $x \in X^\circ$.
By restricting to a suitable curve in $X$ containing $x$, we may further assume that $\dim(X) = 1$.
At this point, we may assume that $\calE$ and $\calF$ are semisimple.
By Corollary~\ref{C:algebraic companion}, $\calE$ admits a semisimple companion $\calF'$ in the same category as $\calF$. On $U$, the restrictions $\calF|_U$ and $\calF|'_U$ are semisimple (by Lemma~\ref{L:restriction functor})
companions of each other; by Theorem~\ref{T:chebotarev}(b), there exists an isomorphism $\calF|_U \cong \calF'|_U$.
By Lemma~\ref{L:restriction functor} again, this isomorphism extends to $X$; this yields the desired result.
\end{proof}

\begin{cor} \label{C:open immersion companion descent}
Let $U$ be an open dense subscheme of $X$. Let $\calE$ and $\calF$ be algebraic coefficient objects on $X$ and $U$, respectively $($not necessarily in the same category$)$, with $\calF$ semisimple. If $\calE|_U$ and $\calF$ are companions, then $\calF$ extends to a coefficient object on $X$ which is a companion of $\calE$.
\end{cor}
\begin{proof}
By Corollary~\ref{C:extend companions}, we only need to check that $\calF$ extends to $X$.
Using the ``cut-by-curves'' criterion (see \cite[Theorem~5.16]{kedlaya-isocrystals} in the crystalline case),
it suffices to check that for any morphism $f: C \to X$ over $k$ where $C$  is a curve over $k$, the pullback $f^* \calF$ extends
from $C \times_X U$ to $C$.
To this end, apply Corollary~\ref{C:algebraic companion} to construct a semisimple crystalline companion $\calF'$ of $f^* \calE$ on $C$. By Lemma~\ref{L:restriction functor}, the restriction of $\calF'$ to $C \times_X U$ is again semisimple.
Now $f^* \calF$ and $\calF'|_{C \times_X U}$ are semisimple companions of each other in the same category,
so by Theorem~\ref{T:chebotarev}(b) they are isomorphic; this proves the claim.
\end{proof}

\begin{cor} \label{C:companion det}
Let $\calE$ and $\calF$ be algebraic coefficient objects on $X$ $($not necessarily in the same category$)$. 
If $\det(\calE)$ is of finite order, then so is $\det(\calF)$.
\end{cor}
\begin{proof}
Since $\det(\calE)$ and $\det(\calF)$ are again companions, we may assume at once that $\calE$ and $\calF$ are of rank 1. Choose a positive integer $n$ for which $\calE^{\otimes n}$ is trivial; then $\calF^{\otimes n}$ is a companion of the trivial object, and so by Theorem~\ref{T:chebotarev}(b)
is itself trivial.
\end{proof}

\begin{cor} \label{C:bound coefficient extension2}
The conclusion of Corollary~\ref{C:bound coefficient extension}, which was originally stated for $\dim(X) = 1$,
holds for general $X$.
\end{cor}
\begin{proof}
We may replace the application of Theorem~\ref{T:companion dimension 1} in the proof of Corollary~\ref{C:bound coefficient extension2} with Theorem~\ref{T:chebotarev}, and then the rest of the argument remains unchanged.
\end{proof}

\begin{remark}
A related assertion to Theorem~\ref{T:chebotarev}(b)
is that $\calE_1$ and $\calE_2$ are forced to have the same semisimplification 
as soon as the equality $P(\calE_{1,x}, T) = P(\calE_{2,x}, T)$ holds for $x$ running over a certain finite subset of $X^\circ$ determined by $X$, the common rank of $\calE_1$ and $\calE_2$, 
and the wild monodromy of $\calE_1$ and $\calE_2$. For the case of curves, this is made explicit by Deligne in \cite[Proposition~2.5]{deligne-finite}; this plays a crucial role in the proof of Theorem~\ref{T:Deligne finite}.
\end{remark}

\begin{remark}
While Theorem~\ref{T:chebotarev} is formulated in terms of an algebraically closed full coefficient field,
it also implies a corresponding assertion for a restricted coefficient field. Namely, suppose that $\calE, \calF$ are two objects in $\Weil(X) \otimes L$ or $\FIsoc^\dagger(X) \otimes L$, where $L$ is a finite extension
of the base coefficient field, which are companions of each other. Suppose further that $\calE$ is irreducible
as a coefficient object (\textit{i.e.}, it remains irreducible after replacing $L$ with any finite extension). 
Then Theorem~\ref{T:chebotarev} implies that $\calE$ and $\calF$ are isomorphic as objects
of $\Weil(X) \otimes \overline{L}$ or $\FIsoc^\dagger(X) \otimes \overline{L}$, but the hypothesis on $\calE$
ensures that the space of homomorphisms from $\calE$ to $\calF$ as coefficient objects is one-dimensional over $\overline{L}$. Moreover, this space is the base extension from $L$ to $\overline{L}$ of the space of homomorphisms 
from $\calE$ to $\calF$ in the original category; consequently, there must exist a nonzero homomorphism from $\calE$ to $\calF$ in the original category, which must then be an isomorphism.
\end{remark}

\subsection{Uniform algebraicity}
\label{subsec:uniform algebraicity}

We next apply the method of Deligne \cite{deligne-finite} 
to address part (ii) of Conjecture~\ref{conj:deligne}.
This discussion constitutes our first use of \emph{uniformity} across curves within a fixed variety, as a technique for studying coefficient objects on higher-dimensional varieties; this technique will prove to be extremely fruitful as we continue.

\begin{lemma} \label{L:extend algebraic}
Let $U$ be an open dense subscheme of $X$, let $E$ be a number field, and let $\calE$ be a coefficient object on $X$.
If $\calE|_U$ is $E$-algebraic, then so is $\calE$.
\end{lemma}
\begin{proof}
Since $E$-algebraicity is a pointwise condition for any given $E$, to check it at some $x \in X^\circ$ it suffices to do so after restriction to some curve in $X$ containing $x$. We may thus assume hereafter that $X$ is a curve over $k$.
We may further assume that $\calE$ is irreducible. 

We first check that $\calE$ is uniformly algebraic (but not necessarily $E$-algebraic). By Lemma~\ref{L:decompose to finite order}, there exists a
constant coefficient object $\calF$ of rank 1 such that $\det(\calE \otimes \calF)$ has finite order.
By restricting to $U$, we see that $\calF^{\otimes \rank(\calE)}$ is (uniformly) algebraic, as then is $\calF$.
Consequently, to check that $\calE$ is uniformly algebraic, we may reduce to the case where $\calE$ is irreducible with finite determinant, in which case Theorem~\ref{T:companion dimension 1} applies.

By the previous paragraph, we can find a Galois number field $E'$ containing $E$ such that $\calE$ is $E'$-algebraic.
By Corollary~\ref{C:Galois conjugates}, for each automorphism $\tau \in \Gal(E'/E)$, we can find another coefficient object $\calE_\tau$ in the same category as $\calE$ whose Frobenius characteristic polynomials are related to those of $\calE$ by application of $\tau$. By hypothesis, $\calE|_U$ and $\calE_{\tau}|_U$ are companions, as then are $\calE$ and $\calE_\tau$ by Corollary~\ref{C:extend companions}. This proves the claim.
\end{proof}

\begin{theorem}\label{T:Deligne finite}
Any algebraic coefficient object on $X$ is uniformly algebraic.
\end{theorem}
\begin{proof}
We use the axiomatized form of Deligne's argument described in \cite[Remarque~3.10]{deligne-finite}
(see also \cite[Theorem~5.1]{esnault-kerz} for an alternate exposition).
By Lemma~\ref{L:extend algebraic}, there is no harm in replacing $X$ with an open dense subscheme during the argument.

Let $\calE$ be an algebraic coefficient object on $X$.
For each positive integer $n$, consider the function $t_n$ assigning to each $x \in X(k_n)$ the trace of $F^n$ on $\calE_x$. These functions have the following properties.
\begin{enumerate}
\item[(a)]
For each morphism $f: Z \to X$ with $Z$ smooth of dimension 1 (but not necessarily geometrically irreducible over $k$), the functions $t_n \circ f$ arise from an \'etale coefficient object $\calF_Z$ on $Z$. This follows from Corollary~\ref{C:algebraic companion}.
\item[(b)]
There exists a connected finite \'etale cover $X'$ of $X$ such that in (a), if $f$ factors through $X'$, then $\calF_Z$ is tame.
In light of Proposition~\ref{P:alter to tame} and Corollary~\ref{C:tame compatibility},
this becomes true after replacing $X$ with a suitable open dense subscheme (but see Remark~\ref{R:semistable reduction} below).
\end{enumerate}
Given these properties, \cite[Remarque~3.10]{deligne-finite} implies the claim.
\end{proof}

\begin{cor} \label{C:uniformly algebraic2}
For $\calE$ a coefficient object on $X$, the following conditions are equivalent.
\begin{enumerate}
\item[(i)]
$\calE$ is algebraic.
\item[(ii)]
$\calE$ is uniformly algebraic.
\item[(iii)]
Each irreducible constituent of $\calE$ is the twist of an object with determinant of finite order by an algebraic object of rank $1$ pulled back from $k$.
\item[(iv)]
There exists a positive integer $n$ such that the pullback of $\calE$ to $X_n$ satisfies (iii) $($with $k$ replaced by $k_n)$.
\end{enumerate}
In other words, the conclusion of Corollary~\ref{C:uniformly algebraic} $($which applies when $X$ is a curve over $k)$ holds for arbitrary $X$. In particular, part (ii) of Conjecture~\ref{conj:deligne} holds.
\end{cor}
\begin{proof}
It is again obvious that (ii) implies (i) and that (iii) implies (iv).
By Lemma~\ref{L:decompose to finite order}, (i) implies (iii).
By Lemma~\ref{L:valuations}(a), (iv)  implies (i).
By Theorem~\ref{T:Deligne finite}, (i) implies (ii). 
\end{proof}

\subsection{\'Etale companions}
\label{subsec:etale companions}

We address part (v) of Conjecture~\ref{conj:deligne} by
adapting the method of Drinfeld \cite{drinfeld-deligne} to include the crystalline case.

\begin{lemma} \label{L:lisse valuation}
Let $\calE$ be an object of $\Weil(X) \otimes \overline{\QQ}_\ell$.
Then $\calE$ is a lisse \'etale $\overline{\QQ}_\ell$-sheaf if and only if for each $x \in X^\circ$, the roots of $P(\calE_x, T)$
all have $\ell$-adic valuation $0$; in particular, this holds if $\calE$ is $p$-plain.
\end{lemma}
\begin{proof}
We may assume that $\calE$ is irreducible. 
If $\calE$ is a lisse \'etale $\overline{\QQ}_\ell$-sheaf, then the image of $W_X$ in the corresponding representation is compact; this implies that the roots of $P(\calE_x, T)$ all have $\ell$-adic valuation $0$.

To check the converse implication, as per Definition~\ref{D:lisse etale}
we may further assume that $\calE$ is of rank 1.
By Lemma~\ref{L:decompose to finite order}, we may write $\calE = \calF \otimes \calL$ where
$\calF$ is of finite order and $\calL$ is constant.
We may then check the claim after replacing $\calE$ with $\calL$; that is, we may assume that $X = \Spec(k)$,
for which the desired statement is obvious.
(Compare \cite[Proposition~3.1.16]{daddezio}.)
\end{proof}

\begin{theorem} \label{T:Weil companion}
Let $\calE$ be a coefficient object on $X$.
\begin{enumerate}
\item[(a)]
If $\calE$ is algebraic, then it has an \'etale companion. 
\item[(b)]
If $\calE$ is irreducible with determinant of finite order, then it has an \'etale companion which is an irreducible lisse \'etale $\overline{\QQ}_\ell$-sheaf with determinant of finite order.
\end{enumerate}
In particular, part (v) of Conjecture~\ref{conj:deligne} holds.
\end{theorem}
\begin{proof}
By Lemma~\ref{L:decompose to finite order}, it suffices to prove (b).
It was shown by Drinfeld \cite[Theorem~2.5]{drinfeld-deligne} that a lisse \'etale $\overline{\QQ}_\ell$-sheaf on $X$ can be constructed out of any family of Frobenius characteristic polynomials at closed points with the following properties
(analogous to the ones used in the proof of Theorem~\ref{T:Deligne finite}).
\begin{enumerate}
\item[(a)]
On any curve in $X$, there is a (not necessarily irreducible) lisse \'etale $\overline{\QQ}_\ell$-sheaf giving rise to the specified characteristic polynomials.
\item[(b)]
The sheaves occurring in (a) obey a uniformity property: their pullbacks along some fixed alteration of $X$ are all tame.
\end{enumerate}
The polynomials $P(\calE_x, T)$ satisfy these conditions: condition (a) follows from Corollary~\ref{C:algebraic companion}, noting that Lemma~\ref{L:valuations}(b) and Lemma~\ref{L:lisse valuation} together guarantee that we get a lisse \'etale  $\overline{\QQ}_\ell$-sheaf, not just a lisse Weil  $\overline{\QQ}_\ell$-sheaf;
condition (b) follows from Proposition~\ref{P:alter to tame} and Corollary~\ref{C:tame compatibility}.

This proves that $\calE$ has an \'etale companion $\calF$, which by Theorem~\ref{T:chebotarev}(a) is again irreducible.
By Corollary~\ref{C:companion det}, $\det(\calF)$ is of finite order.
\end{proof}
\begin{cor} \label{C:companion ell}
Conjecture~\ref{conj:companion} holds when $\ell' \neq p$,
and Theorem~\ref{T:deligne conj2} holds.
\end{cor}
\begin{proof}
Theorem~\ref{T:Weil companion}(a) implies that Conjecture~\ref{conj:companion} holds when $\ell' \neq p$.
As for Theorem~\ref{T:deligne conj2}, 
see Theorem~\ref{T:weight filtration}(a) for part (i);
see Corollary~\ref{C:uniformly algebraic2} for part (ii);
see Lemma~\ref{L:valuations} for parts (iii), (iv);
and see Theorem~\ref{T:Weil companion}(b) for part (v).
\end{proof}

\begin{remark} \label{R:semistable reduction}
As noted above, the semistable reduction theorem for overconvergent $F$-isocrystals \cite[Theorem~7.6]{kedlaya-isocrystals} was originally proved in terms of an uncontrolled alteration, which only becomes finite \'etale over some open dense subscheme of the original variety. However, for $k$ finite, Theorem~\ref{T:Weil companion} makes it possible
(in conjunction with Remark~\ref{R:tamely}) to choose a finite \'etale cover of $X$ with the property that any alteration factoring through this cover suffices to achieve semistable reduction; see \cite[Remark~4.4]{abe-esnault} or \cite[Remark~7.9]{kedlaya-isocrystals}.
\end{remark}

\subsection{Reductions for crystalline companions}
\label{sec:reductions}

One unexpected side benefit of the construction of \'etale companions is that it helps with some key reductions in the construction of crystalline companions. These reductions will be used in \cite{kedlaya-companions2}.

The following is essentially \cite[Proposition~4.3.6]{daddezio}.

\begin{lemma} \label{L:constituent companions}
Let $\calE, \calF$ be two algebraic coefficient objects which are companions.
Then there is a bijection between the irreducible constituents of $\calE$ and those of $\calF$, in which
any two objects in correspondence are again companions.
\end{lemma}
\begin{proof}
Suppose first that at least one of $\calE$ or $\calF$ is an \'etale coefficient object, say $\calF$.
Apply Theorem~\ref{T:Weil companion} to 
each irreducible constituent of $\calE$ to obtain objects in the same category as $\calF$; by Theorem~\ref{T:chebotarev}(a), these objects are all irreducible.
Taking their sum yields obtain another companion $\calF'$ of $\calE$ in the same category as $\calF$;
by Theorem~\ref{T:chebotarev}(b), $\calF$ and $\calF'$ have the same semisimplification. 

Suppose next that $\calE$ and $\calF$ are both crystalline objects. 
By Theorem~\ref{T:Weil companion}, we may insert an \'etale companion $\calG$ between $\calE$ and $\calF$, and then apply the previous paragraph to pass from $\calE$ to $\calG$ to $\calF$.
\end{proof}

\begin{lemma} \label{L:finite etale companion descent}
Let $f: X' \to X$ be a finite \'etale morphism.
\begin{enumerate}
\item[(a)]
Let $\calE$ be an irreducible algebraic coefficient object on $X$,
and let $\calF$ be a semisimple companion of $f^* \calE$. Then 
there exist a companion $\calF_0$ of $\calE$ and an isomorphism $f^* \calF_0 \cong \calF$.
\item[(b)]
Let $\calE$ be a semisimple algebraic coefficient object on $X'$,
and let $\calF$ be a semisimple companion of $f_* \calE$. Then there exists a companion $\calF_0$ of $\calE$
which is a direct summand of $f^* \calF$.
\end{enumerate}
\end{lemma}
\begin{proof}
We first treat (a).
Since $f_*\calF$ is a companion of $f_* f^* \calE$
and $\calE$ arises as a direct summand of $f_* f^* \calE$, we may apply Lemma~\ref{L:constituent companions}
to deduce that $\calE$ has an irreducible companion $\calF_0$ which is a constituent of $f_* \calF$.
By Proposition~\ref{P:global monodromy}, $f^* \calF_0$ is semisimple
(note that $\overline{G}(\calF_0)$ and $\overline{G}(f^* \calF_0)$ have the same connected part thanks to 
Proposition~\ref{P:component group}).
Since $f^* \calF_0$ and $\calF$ are both semisimple companions of $f^* \calE$, by Theorem~\ref{T:chebotarev} they are isomorphic to each other.

We next treat (b). 
By Proposition~\ref{P:global monodromy} again, $f^* \calF$ is semisimple.
Since $f^* \calF$ is a semisimple companion of $f^* f_* \calE$
and $\calE$ occurs as a direct summand of $f^* f_* \calE$, we may again apply Lemma~\ref{L:constituent companions}
to deduce that every constituent of $\calE$ has an irreducible companion which is a direct summand of $f^* \calF$.
This proves the claim.
\end{proof}

\begin{cor} \label{C:alteration companion descent}
Let $f: X' \to X$ be an alteration.
Let $\calE$ be an algebraic coefficient object on $X$.
If $f^* \calE$ admits a crystalline companion, then so does $\calE$.
\end{cor}
\begin{proof}
We may assume from the outset that $\calE$ is irreducible.
Suppose first that $f$ is flat (and hence finite and separable). 
Let $\calF$ be a crystalline companion of $f^* \calE$.
Let $U$ be the open dense subscheme of $X$ over which $f$ is \'etale (and hence finite \'etale because $f$ is proper).
By Lemma~\ref{L:finite etale companion descent}(a), there exist a crystalline companion
$\calF_0$ for $\calE|_U$ and an isomorphism $f^* \calF_0 \cong \calF$. 
This isomorphism gives rise to a descent datum for $\calF$ on the hypercovering of $U$ generated by $f|_U$;
by the full faithfulness of restriction (Lemma~\ref{L:restriction functor}), this descent datum extends over $X$.
By faithfully flat descent for modules, the descent datum arises from an object of $\FIsoc^\dagger(X) \otimes \overline{\QQ}_p$ whose restriction to $U$ is isomorphic to $\calF_0$; by
Lemma~\ref{L:restriction functor} again, $\calF_0$ must itself extend to $X$.
By Corollary~\ref{C:extend companions}, $\calF_0$ is a crystalline companion of $\calE$, as desired.

We now treat the general case.
Let $U$ be the open dense subscheme of $X$ over which $f$ is flat; since $X'$ is smooth,
the complement $Z := X \setminus U$ has codimension at least 2 in $X$.
By the previous paragraph, $\calE|_U$ admits a crystalline companion $\calF$;
by Lemma~\ref{L:restriction functor}, $\calF$ extends uniquely to $X$.
By Corollary~\ref{C:extend companions}, $\calF$ is a crystalline companion of $\calE$, as desired.
\end{proof}

\begin{cor}
To prove Conjecture~\ref{conj:deligne}, it suffices to prove the following: for every $X$ admitting a smooth compactification $\overline{X}$ and every docile \'etale coefficient object $\calE$ on $X$ which is irreducible with finite determinant,
there exists an open dense subscheme $U$ of $X$ such that $\calE|_U$ has a crystalline companion.
\end{cor}
\begin{proof}
By Lemma~\ref{L:restriction functor} (for preservation of irreducibility upon restriction to an open dense subscheme),
Remark~\ref{R:tame in codim 1} (preservation of docileness upon passage to subquotients),
and Corollary~\ref{C:open immersion companion descent}, the hypothesis would imply that every docile $\ell$-adic coefficient object has a crystalline companion.
By Proposition~\ref{P:alter to tame} and Corollary~\ref{C:alteration companion descent}, we would then obtain
Conjecture~\ref{conj:companion} for $\ell' = p$. By Corollary~\ref{C:companion det} and Lemma~\ref{L:constituent companions},
the crystalline companion of an \'etale coefficient object which is irreducible of finite determinant is again  irreducible of finite determinant. In light of Theorem~\ref{T:deligne conj2}, this proves the claim.
\end{proof}

\subsection{Comparison with Abe--Esnault and Lefschetz slicing}
\label{sec:abe-esnault}

As mentioned above, there is a strong overlap between the results described in this section and those of Abe--Esnault
\cite{abe-esnault}; the main differences are in the handling of Lefschetz slicing. We first indicate the main differences, then indicate how to reconcile the results.

\begin{remark}
For tame crystalline coefficient objects, Abe--Esnault prove a form of the slicing principle 
without recourse to weights \cite[Theorem~2.5]{abe-esnault}.
This then yields the existence of \'etale companions for tame crystalline coefficient objects
\cite[Proposition~2.7]{abe-esnault} using the arguments of Deligne and Drinfeld described above.

Abe--Esnault then use the existence of \'etale companions in the tame case to establish a form of the slicing principle without the tame restriction \cite[Theorem~3.10]{abe-esnault}. This in turn yields the existence of \'etale companions
as above.
\end{remark}

\begin{remark} \label{R:slicing a posteriori}
Note that the forms of the slicing principle obtained in \cite{abe-esnault} are \emph{a priori} stronger than Lemma~\ref{L:slicing}, in that there is no base extension necessary. On the other hand, with Theorem~\ref{T:deligne conj2} in hand, one can transfer \emph{any} version of the Lefschetz slicing principle from the \'etale case to the crystalline case; in particular, one can recover these stronger statements \emph{a posteriori} by working on the \'etale side, as demonstrated by Theorem~\ref{T:strong slicing} and Theorem~\ref{T:strong slicing2} below.
However, even in the \'etale case some care is required; for example, an incorrect version of the slicing principle is asserted in \cite[VII.7]{lafforgue}. See \cite[Theorem~8.1]{esnault-lefschetz} for further discussion of this point,
and the rest of that article for a broader discussion of Lefschetz slicing.
\end{remark}

The following result reproduces \cite[Theorem~2.5]{abe-esnault},
modulo a theorem of Poonen \cite[Theorem~1.1]{poonen} which can be used to produce curves as described passing through any finite set of closed points. Compare \cite[Proposition~C.1]{drinfeld-deligne}.
(See \cite[Lemma~5.4]{esnault-lefschetz} for an example showing that the tame hypothesis cannot be omitted.)

\begin{theorem} \label{T:strong slicing}
Let $U$ be an open dense affine subscheme of $X$.
Let $H_1,\dots,H_{n-1}$ be very ample divisors on $U$ which intersect transversely.
Let $\calE$ be a tame irreducible coefficient object on $X$.
Then the restriction of $\calE$ to $C := H_1 \cap \cdots \cap H_{n-1}$ is irreducible.
\end{theorem}
\begin{proof}
By Lemma~\ref{L:restriction functor}, $\calE|_U$ is irreducible; we may thus assume hereafter that $U=X$.
Using Lemma~\ref{L:decompose to finite order}, we may reduce to the case where $\calE$ has determinant of finite order;
by Corollary~\ref{C:uniformly algebraic2}, $\calE$ is algebraic. 
By Theorem~\ref{T:Weil companion} and Lemma~\ref{L:constituent companions}, $\calE$ admits an irreducible $\ell$-adic companion for any $\ell \neq p$; we may thus further reduce to the \'etale case. In this setting, the claim follows from the fact that for any connected finite \'etale cover $W$ of $X$ which is tamely ramified at the boundary, $W \times_{X} C$ is also connected;
this follows from the Bertini connectedness theorem \cite[Theorem~2.1(A)]{fulton-lazarsfeld}
as in \cite[Lemma~C.2]{drinfeld-deligne}.
\end{proof}

We now turn around and invoke a lemma from \cite{abe-esnault} that allows one to formally upgrade
Theorem~\ref{T:strong slicing}. 
\begin{lemma}[Abe--Esnault] \label{L:ae same mono}
Let $f: Y \to X$ be a morphism of smooth connected $k$-schemes.
Let $\calE$ be a coefficient object on $X$. Suppose that for each positive integer $n$,
for $\calF := (\calE \oplus \calE^\dual)^{\otimes n}$, the map
$H^0(X, \calF) \to H^0(Y, f^* \calF)$ is an isomorphism.
Then the morphism $\overline{G}(\calE) \to \overline{G}(f^* \calE)$ $($defined with respect to a consistent choice of base points$\,)$ is an isomorphism.
\end{lemma}
\begin{proof}
This is essentially \cite[Proposition~1.9]{abe-esnault} modulo a slight change of hypotheses; we sketch the argument to show that nothing is affected. Let $[\calE]$ be the Tannakian category of geometric coefficient objects generated by $\calE$. By hypothesis, $f^*$ restricts to a fully faithful functor on $[\calE]$. This is not enough to prove the claim; however, a quick group-theoretic argument \cite[Lemma~1.7]{abe-esnault} shows that this does hold provided that every $\calG \in [\calE]$ of rank 1 is of finite order. By Remark~\ref{R:absolutely irreducible}, $\calG$ promotes to a true coefficient object on $X_n$ 
for some $n$; we may thus apply Lemma~\ref{L:decompose to finite order} to deduce the claim.
\end{proof}

\begin{cor} \label{C:same monodromy}
With notation as in Theorem~\ref{T:strong slicing}, the induced morphism $\overline{G}(\calE) \to \overline{G}(\calE|_C)$ $($defined with respect to a consistent choice of base points$)$ is an isomorphism.
\end{cor}
\begin{proof}
By Theorem~\ref{T:strong slicing}, for every constituent of
$(\calE \oplus \calE^{\dual})^{\otimes n}$ as a coefficient object, the pullback to $C$ is again an
irreducible coefficient object.
By Lemma~\ref{L:decompose to finite order2}, the same is true for geometric coefficient objects.
We may thus apply Lemma~\ref{L:ae same mono} to conclude.
\end{proof}

On a related note, we mention the following result.
\begin{cor} \label{C:finiteness}
For any positive integer $r$, any category of coefficient objects only contains finitely many isomorphism classes of irreducible tame coefficient objects of rank $r$ on $X$ up to twisting by constant objects of rank $1$.
\end{cor}
\begin{proof}
By Theorem~\ref{T:Weil companion}, it suffices to check the case of a category of \'etale coefficient objects.
By Lemma~\ref{L:restriction functor}, we may assume that $X$ is affine.
By Corollary~\ref{C:same monodromy} and the aforementioned theorem of Poonen, we can find a curve $C$ in $X$
such that restriction of coefficient objects from $X$ to $C$ is fully faithful;
by Corollary~\ref{C:finiteness curve2}, we deduce the claim.
\end{proof}

Without the tameness hypothesis, one has the following result which reproduces \cite[Theorem~3.10]{abe-esnault}.

\begin{theorem} \label{T:strong slicing2}
Let $\calE$ be an irreducible coefficient object on $X$.
There is a dense open subscheme $U$ of $X$ such that for any finite subset $S \subseteq U^\circ$,
there exist a curve $C$ and a morphism $f: C \to X$ over $k$ such that $f$ admits a section on $S$
and $f^* \calE$ is irreducible.
\end{theorem}
\begin{proof}
We may reduce to the \'etale case as in the proof of Theorem~\ref{T:strong slicing};
then \cite[Theorem~8.1]{esnault-lefschetz} gives the desired result.
(This amounts to applying Hilbert irreducibility in the form of \cite[Theorem~2.15]{drinfeld-deligne}.)
\end{proof}

\subsection{On a conjecture of de Jong}

The following conjecture is due to de Jong, as reported in 
\cite{esnault-shiho, shiho-fg} where some special cases are established.
\begin{conj}[de Jong]
Let $X$ be a smooth, proper, geometrically connected variety over a $($not necessarily finite$)$ perfect field $k$ of characteristic $p$ with trivial geometric \'etale fundamental group.
Then every convergent isocrystal on $X$ is constant.
\end{conj}

Using \'etale companions, one obtains the following result in the direction of this conjecture.
(Compare \cite[Remark~4.5]{abe-esnault}.)
\begin{prop} \label{P:de jong conj}
Suppose that $(k$ is finite and$\,)$
$X$ is proper over $k$ and has trivial geometric \'etale fundamental group.
Then the functor 
\[
2\mbox{-}\!\!\varinjlim_{n \to \infty} \FIsoc^\dagger(\Spec k_n) \otimes \overline{\QQ}_p \to 2\mbox{-}\!\!\varinjlim_{n \to \infty} \FIsoc^\dagger(X_n) \otimes \overline{\QQ}_p
\]
is an equivalence of categories.
\end{prop}
\begin{proof}
We are free to enlarge $k$, so we may assume that $X$ contains a $k$-rational point; we may also assume that $X$ is irreducible.
Since $X$ contains a $k$-rational point, the functor $\FIsoc^\dagger(\Spec k) \otimes \overline{\QQ}_p \to \FIsoc^\dagger(X) \otimes \overline{\QQ}_p$ is fully faithful; moreover, it preserves extension groups.
Consequently, it remains to check that the essential image contains an arbitrary irreducible $\calE \in \FIsoc^\dagger(X) \otimes \overline{\QQ}_p$, possibly after enlarging $k$.
At this point, there is no harm in twisting by a constant rank-1 object; by Lemma~\ref{L:decompose to finite order},
we may thus reduce to the case where $\det(\calE)$ is of finite order.
By Theorem~\ref{T:Weil companion}, $\calE$ arises as a crystalline companion
of some lisse Weil $\overline{\QQ}_{\ell}$-sheaf. By hypothesis, the latter
must arise by pullback from $k$, as then must $\calE$ by Theorem~\ref{T:chebotarev}(b).
\end{proof}


\end{document}